\documentclass[onefignum, onetabnum]{siamart220329}

\usepackage{amsfonts}
\usepackage{amsmath, amssymb, mathtools}
\usepackage{cases}				
\usepackage{esint}				
\usepackage{dsfont}				
\usepackage{enumitem}
\usepackage{tikz}
\ifpdf
	\DeclareGraphicsExtensions{.eps,.pdf,.png,.jpg}
\else
	\DeclareGraphicsExtensions{.eps}
\fi

\newcommand{\fref}[2]{\hyperref[#2]{#1~\ref*{#2}}} 
\newcommand{\norm}[2]{{\lvert \lvert #1 \rvert \rvert}_{#2}} 
\DeclareMathOperator{\id}{id} 
\DeclareMathOperator*{\argmin}{arg\,min} 
\DeclareMathOperator*{\sign}{sign} 
\DeclareMathOperator*{\supp}{supp} 
\newcommand{\PV}{PV_e} 

\newsiamremark{remark}{Remark}
\newsiamremark{example}{Example}
\newsiamthm{prop}{Proposition}
\newsiamthm{cor}{Corollary}

\headers{Well-posedness and regularity at cloud edge}{A. Remond-Tiedrez, L. Smith, and S. Stechmann}

\title{
	A nonlinear elliptic PDE from atmospheric science: well-posedness and regularity at cloud edge
}

\author{
	Antoine Remond-Tiedrez			\thanks{Department of Mathematics, University of Wisconsin--Madison, Madison, WI, USA}
	(\email{aremondtiedrez@math.wisc.edu})
	\and Leslie M. Smith			\thanks{Department of Mathematics and Department of Engineering Physics, University of Wisconsin--Madison, Madison, WI, USA}
	(\email{lsmith@math.wisc.edu})
	\and Samuel N. Stechmann			\thanks{Department of Mathematics and Department of Atmospheric and Oceanic Sciences, University of Wisconsin--Madison, Madison, WI, USA}
	(\email{stechmann@wisc.edu})
}

\usepackage{amsopn}

\ifpdf
\hypersetup{
	colorlinks=true,
	linkcolor=RoyalBlue,
	citecolor=Thistle,
	pdftitle={A nonlinear elliptic PDE from atmospheric science: well-posedness and regularity at cloud edge},
	pdfauthor={A. Remond-Tiedrez, L. Smith, and S. Stechmann}
}
\fi

\begin{document}


\newpage

\maketitle


\begin{abstract}
	The precipitating quasi-geostrophic equations go beyond the (dry) quasi-geostrophic equations by incorporating the effects of moisture.
	This means that both precipitation and phase changes between a water-vapour phase (outside a cloud) and a water-vapour-plus-liquid phase (inside a cloud) are taken into account.
	In the dry case, provided that a Laplace equation is inverted,
	the quasi-geostrophic equations may be formulated as a nonlocal transport equation for a single scalar variable (the potential vorticity).
	In the case of the precipitating quasi-geostrophic equations, inverting the Laplacian is replaced by a more challenging adversary known as potential-vorticity-and-moisture inversion.
	The PDE to invert is nonlinear and piecewise elliptic with jumps in its coefficients across the cloud edge.
	However, its global ellipticity is a priori unclear due to the dependence of the phase boundary on the unknown itself.
	This is a free boundary problem where the location of the cloud edge is one of the unknowns.
	Here we present the first rigorous analysis of this PDE,
	obtaining existence, uniqueness, and regularity results. 
	In particular the regularity results are nearly sharp.
	This analysis rests on the discovery of a variational formulation of the inversion.
	This novel formulation is used to answer a key question for applications: which quantities jump across the interface and which quantities remain continuous?
	Most notably we show that the gradient of the unknown pressure, or equivalently the streamfunction, is H\"{o}lder continuous across the cloud edge.
\end{abstract}






\vfill
\noindent
\textbf{Table of contents}

\vspace{1em}
\begin{enumerate}
	\item \hyperref[sec:intro]			{Introduction}							\hfill \pageref{sec:intro}
	\item \hyperref[sec:discuss]			{Discussion}							\hfill \pageref{sec:discuss}
	\item \hyperref[sec:background]			{Background}							\hfill \pageref{sec:background}
	\item \hyperref[sec:var_form_and_prop_en]	{Variational formulation and basic properties of the energy}	\hfill \pageref{sec:var_form_and_prop_en}
	\item \hyperref[sec:wp]				{Well-posedness}						\hfill \pageref{sec:wp}
	\item \hyperref[sec:smoothing]			{Smoothing}							\hfill \pageref{sec:smoothing}
\end{enumerate}

\vspace{1em}
Appendix
\begin{enumerate}[label=\Alph*.]
	\item \hyperref[sec:tools_convex_ana]		{Tools from convex analysis}					\hfill \pageref{sec:tools_convex_ana}
	\item \hyperref[sec:tools_reg_th]		{Tools from regularity theory}					\hfill \pageref{sec:tools_reg_th}
	\item \hyperref[sec:rosetta]			{Rosetta stone}							\hfill \pageref{sec:rosetta}
	\item \hyperref[sec:comparison]			{Comparison of the variational and conserved energies}		\hfill \pageref{sec:comparison}
	\item \hyperref[sec:pqg]			{Precipitating quasi-geostrophic equations}			\hfill \pageref{sec:pqg}
\end{enumerate}
\vfill
\newpage


\textbf{Note to the reader:}
\fref{Section}{sec:intro} acts as a ``shortest path'' to the statements of the main result of this paper.
\fref{Section}{sec:discuss} discusses the main difficulties encountered, provides a map of the remainder of the paper and its arguments,
and situates the results herein with respect to the broader literature.  More details of the model itself are presented in 
\fref{Section}{sec:background} since the model is not well-known to the mathematical community.
Finally \fref{Sections}{sec:var_form_and_prop_en}--\ref{sec:smoothing} are the main body of the paper and contain the arguments culminating in the central theorems stated in \fref{Section}{sec:intro} below.

\section{Introduction and statement of the results}
\label{sec:intro}
An important success story in atmospheric dynamics is the description of incompressible fluids in the limit of fast rotation and strong stratification.
This limit is relevant on the synoptic scale\footnote{
The word ``synoptic'' has Greek roots: ``syn--'' means ``together'' and ``--optic'' means ``seen''.
The synoptic scale is thus that at which a ``general picture'' emerges.
In concrete terms it usually refers to a spatial scale of about 1,000 km and a temporal scale of a few days.
}
for the description of the troposphere and was introduced by Charney in \cite{charney} as the \emph{quasi-geostrophic} (QG) model \cite{majda_2003, vallis_2017}.
One of its key successes is that it reduces the dynamics to a single scalar variable, namely the \emph{potential vorticity} which we will denote as $PV$ \cite{ertel}.
Crucially for our discussion here: phrasing the entire QG 
dynamics in terms of the
$PV$ requires the inversion of a Laplacian
-- this is known as potential vorticity inversion, or \emph{$PV$ inversion} \cite{hoskins_85}.
This framework is similar to the vorticity formulation of the incompressible Euler equations where the Biot-Savart law requires the inversion of a Laplacian.

The classical QG
model does not take into account moisture, despite the key role that moisture plays in atmospheric dynamics,
most notably because water underpins important mechanisms for energy transfer in the atmosphere.
As water evaporates it absorbs heat, while as water condensates it releases the latent heat it had stored.
This motivates consideration of \emph{moist} variants of the QG
system.

In this paper we will consider one such moist model: the \emph{precipitating quasi-geostrophic} (PQG) model introduced in \cite{PQG2017}.
In that model the dynamics are now fully described by two variables: an \emph{equivalent} potential vorticity $PV_e$ which differs slightly from the one discussed above
(more on that in \fref{Section}{sec:background}) and a moist scalar variable $M$.
By contrast with the classical case in which moisture is not accounted for and where $PV$ inversion requires the inversion of a Laplacian,
when moisture is involved we now have to deal with $PV_e$-and-$M$ inversion, which relies on the inversion of a more complicated operator.
The complications involve both nonlinearities and lack of smoothness at cloud edge.
The object of this paper is the first rigorous analysis of this operator and of $PV_e$-and-$M$ inversion.

We need to discuss one more aspect of the model before we can write down $PV_e$-and-$M$ inversion, namely the manner in which the moisture is incorporated in the model.
The moisture is included in the model via a single scalar unknown $q$ which describes the total water content (water vapour plus condensed water).
For simplicity we will assume without loss of generality that saturation occurs at $q=0$.

Since different physics will occur depending on the sign of $q$, we introduce the following notation for convenience.
For the unsaturated phase $ \left\{ q<0 \right\}$ which contains water only in vapour form we introduce a Heaviside function $H_u = \mathds{1} (q<0)$,
while for the saturated phase $ \left\{ q \geqslant 0 \right\}$ which also contains water in the form of raindrops we introduce $H_s = 1-H_u = \mathds{1}(q\geqslant 0)$
-- see also \fref{Figure}{fig:phases}.
\begin{figure}
	\centering
	\includegraphics{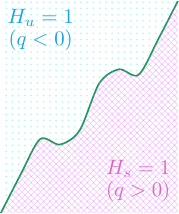}
	\caption{
		A pictorial depiction of the interface separating the two phases of water considered here
		(recall that, without loss of generality, saturation is taken to occur at $q=0$).
		The blue region on the left  corresponds to the unsaturated region where $H_u = 1$ (the outside of a cloud)
		whereas the pink region on the left where $H_s = 1$ corresponds to the saturated region (the inside of a cloud).
		The level set $ \left\{ q = 0 \right\}$ is the cloud edge.
		Recall that $H_u = \mathds{1} (q<0)$ and $H_s = \mathds{1} (q \leqslant 0)$ such that $H_u + H_s = 1$.
	}
	\label{fig:phases}
\end{figure}
$PV_e$-and-$M$ inversion then takes the following form,
\begin{equation}
\label{eq:PV_M}
	\Delta_h p + \partial_3 \left( \frac{1}{2} \left( M + \partial_3 p \right)  H_u + \left( \partial_3 p \right) H_s \right) = PV_e,
\end{equation}
where $PV_e$ and $M$ are given and we seek to solve for the pressure $p$.\footnote{
	In the precipitating quasi-geostrophic 
 system the pressure acts as a streamfunction (as in the dry case).
	Solving for the pressure thus means solving for the velocity.
	Since the velocity transports the active scalars $PV_e$ and $M$, which in turn determine the pressure through $PV_e$-and-$M$ inversion, the system is closed.
	This is discussed in more detail in \fref{Section}{sec:background}.
}
Note that throughout this paper we will work on the $3$-torus $\mathbb{T}^3,$
and we use the
notation $\Delta_h = \partial_1^2 + \partial_2^2$,
where $x_3$ denotes the direction of both rotation and stratification (this is the so-called ``traditional approximation'' \cite{eckart}).
The PDE (\ref{eq:PV_M}) is nonlinear since $H_u$ and $H_s$ are determined by the water content $q$, which is in turn a function of $M$ and $\partial_3 p$
(see \fref{Section}{sec:background} for a more detailed discussion of the 
precipitating quasi-geostrophic model,
which includes a derivation of \eqref{eq:PV_M} and more details on how $M$ and $\partial_3 p$ determine $q$).
Immediately, \eqref{eq:PV_M} introduces two key questions.

\textbf{Question 1:}
Is equation (\ref{eq:PV_M}) truly elliptic, and thus solvable?
Certainly it is elliptic in each of the saturated and unsaturated phases, but how does the dependence of the phase boundary on the unknown $p$
impact the overall ellipticity and solvability of the problem?

A first attempt at answering this would seek to smooth out the Heavisides appearing in \eqref{eq:PV_M}.
This raises subtle issues discussed 
in \fref{Remark}{rmk:agnostic_smoothing}.

\textbf{Question 2:}
What quantities, if any, are guaranteed to be continuous across the interface between the two phases?
Conversely, are there examples for which certain quantities jump across the interface?

Our ability to answer these two questions rests on the following discovery.
We identify a variational structure underpinning $PV_e$-and-$M$ inversion.
Note that such a structure is not always present for piecewise elliptic systems and relies critically on the specific form of $PV_e$-and-$M$ inversion.
We postpone a more detailed discussion of the variational structure to \fref{Section}{sec:var_form_and_prop_en}.
Here we note that we
may reformulate $PV_e$-and-$M$ inversion by invoking the minimum function instead of explicitly using Heavisides
to obtain
\begin{equation}
\label{eq:PV_M_min0}
	\Delta p + \frac{1}{2} \partial_3 \left( \min \left( M - \partial_3 p,\, 0 \right) \right) = PV_e.
\end{equation}
This is obtained by paying close attention to how the water content $q$ is related to the moist variable $M$ and the unknown pressure $p$ (this is discussed in detail in \fref{Section}{sec:background}).

We will use the notation involving the minimum function throughout this paper but we could alternatively have used the fact that
$f^- = - \min (f,\,0)$, where $f^-$ denotes the negative part of a function $f$, to phrase $PV_e$-and-$M$ inversion in terms of the negative part of $M - \partial_3 p$.
The formulation \eqref{eq:PV_M_min0} of $PV_e$-and-$M$ inversion allows us to answer both of the key questions above,
as shown by the well-posedness result stated below.
Note that all functional spaces are assumed to be taken over the torus $\mathbb{T}^3$ unless specified otherwise.
Moreover the main spaces of interest will be $L^2 \vcentcolon= \left\{ u : \int_{\mathbb{T}^3} u^2 < \infty \right\}$,
$ \mathring{H}^1~\vcentcolon=~\left\{ u \in L^2 : \nabla u \in L^2 \text{ and } u \text{ has average zero} \right\}$,
and its dual $H^{-1} \vcentcolon= {\big( \mathring{H}^1 \big)}^*$.
\begin{theorem}
\label{thm:wp}
	For every $M\in L^2$ and $PV_e \in H^{-1}$ there exists a unique weak solution $p\in \mathring{H}^1$ of \eqref{eq:PV_M}.
	Moreover there exists $\alpha\in (0,\,1)$ such that if $PV_e$ and $M$ are smooth then
	\begin{enumerate}
		\item	$p\in C^{1,\alpha}$ globally, including across the interface, in the sense that the gradient of $p$ is H\"{o}lder continuous with exponent $\alpha$ and
		\item	both the unsaturated and saturated phase, $ \left\{ q < 0 \right\} $ and $ \left\{ q > 0 \right\}$ respectively, are open sets and $p$ is smooth in each phase.
	\end{enumerate}
	Finally: this regularity result is nearly sharp in the sense that there exist smooth $PV_e$ and $M$ for which the unique solution $p$ of \eqref{eq:PV_M} belongs to $C^{1,\,1}$
	but not $C^2$, i.e. the gradient of $p$ is Lipschitz but not differentiable.
\end{theorem}

While question 2 above asked which quantities remain continuous or jump across the interface, \fref{Theorem}{thm:wp} above only mentions the pressure $p$.
Nonetheless this result answers question 2 since all other quantities of interest may be reconstructed from the pressure using geostrophic and hydrostatic balance
-- this is discussed in \fref{Section}{sec:background}.
In particular we deduce from these balances that the velocity $u$, the equivalent potential temperature $\theta_e$, and the total water content $q$ have the same regularity as the gradient of the pressure:
they are H\"{o}lder continuous across the interface. (Note that the variables $u$ and $\theta_e$ are also discussed in \fref{Section}{sec:background}.)


\section{Discussion: difficulties, overview, and related literature}
\label{sec:discuss}

The main difficulty in the study of $PV_e$-and-$M$ inversion in its native form \eqref{eq:PV_M} is the presence of the Heavisides.
From a modelling perspective it is remarkably useful to invoke the Heavisides $H_u$ and $H_s$.
They allow us to clearly see how 
physical relations depend on water according to the phase in which they occur.
From a mathematical perspective, however, they introduce an ouroboros-esque danger of eating your own tail.
Given fixed Heavisides, \eqref{eq:PV_M} is piecewise elliptic and so its solvability
seems all but assured. However, these Heavisides are not fixed and depend on the solution.
Here lies the difficulty: the solution is determined by the Heavisides, which themselves are determined by the solution.
The snake is eating its own tail.
This is the typical issue that arises when tackling free boundary problems.

On the other hand,
rewriting \eqref{eq:PV_M} in the form of \eqref{eq:PV_M_min0} 
allows for progress because
it folds the free boundary 
$ \left\{ M=\partial_3 p \right\}$ explicitly into the nonlinearity.
This means that solving the PDE requires no \emph{a priori} knowledge of the free boundary: we may simply find the solution via a variational characterization and
\emph{then} determine the location of the free boundary from the solution. The snake's tail is out of its mouth.

To be more precise, reformulating $PV_e$-and-$M$ inversion in the form of \eqref{eq:PV_M_min0} helps us in two ways.
First, the formulation \eqref{eq:PV_M_min0} leads more directly to a variational formulation of $PV_e$-and-$M$ inversion.
The Direct Method of the Calculus of Variations may then be brought to bear to deduce the first part of \fref{Theorem}{thm:wp},
as shown
in \fref{Section}{sec:wp}, culminating in \fref{Theorem}{thm:exist_and_unique}.
In particular, 
we provide
one way of answering our initial question asked immediately after first encountering $PV_e$-and-$M$ inversion in its form \eqref{eq:PV_M} involving Heavisides:
demonstration that
the energy is strongly convex (see \fref{Proposition}{prop:prop_en}) verifies that indeed $PV_e$-and-$M$ inversion is an elliptic problem.

Note that there is another candidate for the energy characterising $PV_e$-and-$M$ inversion.
This alternate candidate is arguably more natural since it corresponds to the energy conserved by the PQG dynamics.
However, for reasons detailed in \fref{Remark}{rmk:comparison_en}, this conserved energy is not suitable and thus highlights the care required to identify an appropriate variational principle.

The second way in which the formulation \eqref{eq:PV_M_min0} helps us is with regard to regularity.
Indeed, while differentiating \eqref{eq:PV_M} would prove challenging due to the presence of Heavisides, differentiating that same equation in the form of \eqref{eq:PV_M_min0} is now straightforward.
It tells us that $ u = \partial_i p$ solves
\begin{equation}
	\label{eq:diff_PV_M}
	\partial_i PV_e = \Delta u - \frac{1}{2} \partial_3 \left( \mathds{1}(M < \partial_3 p) \partial_3 u \right).
\end{equation}
In particular, since the coefficient matrix $A = I - \frac{1}{2} H_u e_3\otimes e_3$ is uniformly elliptic, we may use the tools from de Giorgi's regularity theory
to deduce the regularity result for $PV_e$-and-$M$ inversion recorded in the second part of \fref{Theorem}{thm:wp}.
Here we have used the notation ${(v \otimes w)}_{ij} = v_i w_j$ for any two vectors $v$ and $w$.
We note that the regularity theory presented in \fref{Section}{sec:wp} is difficult to find in detail in the literature. References
\cite{vasseur} and \cite{cafarelli_vasseur} are excellent introductions to the simplest setting in which de Giorgi tools bear fruition, when there is no forcing,
and \cite{han_lin} contains a version of the regularity theory presented here but without proof.
The argument presented in detail here emphasizes how to adapt the classical de Giorgi result to the case of nonzero forcing
through the use of the Campanato characterisation of H\"{o}lder continuity.
Note also that \eqref{eq:diff_PV_M} provides another way of answering our first question of whether or not $PV_e$-and-$M$ inversion is truly elliptic:
the ellipticity of the coefficient matrix $A$ discussed above answers that question in the affirmative.
A final remark on the regularity theory is that it is of key importance for applications.
Indeed, one often seeks to identify the quantities which jump across the interface and the quantities which remain continuous across that interface.
The regularity theory we develop allows us to identify these quantities, most notably determining that the gradient of the pressure is H\"{o}lder continuous across the interface.

In the last section of this paper we provide some computations on how to smooth out $PV_e$-and-$M$ inversion.
These computations help understand why analysing the ellipticity of $PV_e$-and-$M$ inversion directly from \eqref{eq:PV_M} is quite subtle, and they are also of independent interest.
Indeed, they provide fertile ground for further analysis: when solving $PV_e$-and-$M$ inversion numerically, it may be beneficial to first solve the smoothed-out version,
and then let the smoothing parameter decrease to zero.
It also remains to be seen whether the smoothed out interfaces are close to the true interface -- this would be guaranteed if we had $C^{1,\,\alpha}$ estimates on the distance between
the smoothed out minimizers and the true minimizer.
As a first step we provide $H^1$ estimates here.

Ultimately, this paper is also intended to bring the attention of the PDE community to a free boundary problem not yet studied.
At the end of this section we discuss how this problem differs from those most commonly studied in the literature.
The well-posedness result provided here is but the first step in the more careful analysis of this particular free boundary problem,
and also serves to give confidence in
numerical studies of the PQG system that rely on $\PV$-and-$M$ inversion (such as \cite{hu_edwards_smith_stechmann_21}).
Moreover, important open questions remain.
For instance we do not yet know anything about the regularity of the interface, even conditionally based on assumptions on the regularity and/or compatibility of $PV_e$ and $M$.

The rigorous analysis of moist atmospheric models has only recently attracted the attention of the PDE community.
As we review some of that work, we focus on the rigorous treatment of models involving phase transitions
(some earlier works studied moist quantities that were purely advected by the velocity field, without phase changes).

The first rigorous foray into moist atmospheric dynamics with phase changes can be traced back to a series of work
\cite{coti_zelati_temam_12, coti_zelati_fremond_temam_tribbia_13, bousquet_coti_zelati_temam_14}
which studies a two-phase (water vapour and condensed water) model where the velocity field is taken to be given.
The velocity is then evolved dynamically via the primitive equations in \cite{coti_zelati_huang_kukavica_temam_ziane}
and topography is considered in \cite{lian_ma}. The same model, but with non-constant water saturation ratio, is studied in \cite{temam_wu_15, teman_wang_16}.

There are then two branches of study for \emph{three}--phase models (which involve water vapour and distinguish two kinds of condensed water: cloud water and precipitating water).\footnote{Note that the version of the PQG model used here is a two-phase version in which the condensed water phase can be cloud water that does not fall, or precipitating water (rain) with prescribed fall speed, but the version here does not consider co-existence of clouds and rain or the conversion from cloud water to rain water. Other versions of precipitating quasi-geostrophic models, with three phases of water or other additional complexities in cloud microphysics, can also be defined; see discussion in \cite{PQG2017,wetzel_smith_stechmann_19,wetzel_smith_stechmann_20}.}
One branch relies on cloud microphysics from \cite{grabowski}, beginning with \cite{cao_hamouda_temam_tribbia} where the velocity is given.
In \cite{tan_liu} the velocity then solves the primitive equations.
Another branch relies on cloud microphysics from \cite{klein_majda_06}, beginning with \cite{hittmeir_klein_li_titi_17} where the velocity is given.
In \cite{hittmeir_klein_li_titi_20} the velocity solves the primitive equations instead.

The incorporation of additional phases due to ice is discussed in \cite{cao_jia_temam_tribbia}.
By contrast, earlier works which omit cloud ice are thus sometimes referred to as ``warm cloud'' models.

It is important to note that all works mentioned above either treat the velocity as prescribed or as governed by the primitive equations.
Following \cite{cao_titi}, a well-understood framework relying on the barotropic--baroclinic decomposition of the velocity is in place for the analysis of the primitive equations and its extensions.
By contrast, and to the best of the authors' knowledge, there are as of yet no rigorous mathematical treatment of moist models leveraging other dynamical models for the velocity.

One such dynamical model is the precipitating quasi-geostrophic system discussed here.
In order for this system to be well-formulated as a nonlocal transport system (let alone well-posed, since that will be the object of future work), $\PV$-and-$M$ inversion needs to be understood.
We thus provide here the first building block in a framework amenable to the rigorous analysis of another dynamical model for moist atmospheric dynamics besides the primitive equations.

Free boundary problems have by contrast garnered significant interest from the mathematical community.
Rather than provide an extensive review of the
literature,
we direct the reader to 
\cite{chen_shahgholina_vazquez, figalli_shahgholia}
and references therein.
Instead, we only highlight the key difference (and thus sources of interest) of $\PV$-and-$M$ inversion compared with oft-studied free boundary problems in the literature.  Here
the free boundary depends on the \emph{gradient} of the unknown, and not on the unknown itself.  The latter 
case is
the more typical free boundary problem known as the obstacle problem (see the book \cite{petrosyan_shahgholian_uraltseva}).

Note that the dependence of the free boundary on the gradient is not unseen in the literature --
it is studied for example as ``Model Problem B'', one of three instances of standard classes of free boundary problems studied in \cite{petrosyan_shahgholian_uraltseva}.
That being said, this ``Model Problem B'' is isotropic since the free boundary is the boundary of the set $ \left\{ \lvert \nabla u \rvert > 0 \right\} $,
whereas the problem of interest here is \emph{anisotropic} since
the \emph{vertical} derivative of $p$ 
determines the free boundary.
It is also worth noting that, although not unseen in the literature, gradient-dependent problems are less prominent.

\section{Background and derivation of $PV_e$-and-$M$ inversion}
\label{sec:background}

In this section we discuss the PQG equations in more detail since that model is recent and not yet well-known in the mathematical community.
In particular we provide a derivation of $PV_e$-and-$M$ inversion in both its form \eqref{eq:PV_M} and \eqref{eq:PV_M_min0}.
We omit any considerations of the dynamics of the PQG equations here, which is a rich and complex problem in and of itself.
For the sake of completeness we record the dynamic equations in \fref{Appendix}{sec:pqg}.

In order to motivate the form taken by $PV_e$-and-$M$ inversion, we will first discuss how $PV$ inversion arises in the dry case.
We will then see how these arguments are modified once moisture is incorporated as per the PQG model introduced in \cite{PQG2017}.
This will enable us to make the form of $PV_e$-and-$M$ inversion precise and to state the main results we obtain pertaining to the solvability of that inversion.

\textbf{The dry case.}
We begin with a discussion of the dry QG model. 
While the complete dynamics may be characterised solely in terms of the potential vorticity $PV,$
it is 
helpful to describe the system by introducing the
auxiliary unknowns: the
velocity $u$, the potential temperature\footnote{
	The potential temperature is a quantity defined in terms of the temperature and the (compressible) pressure.
	It is often more convenient to work with potential temperature than temperature itself (as is the case for dry QG) since, unlike the temperature,
	the \emph{potential} temperature is purely advected by the flow.
}
$\theta$, 
and the pressure $p$.

With these auxiliary variables in hand we can view the quasi-geostrophic system as a combination of three features. The first feature is that the potential vorticity defined as
$
	PV = {(\nabla\times u)}_3 + \partial_3 \theta
$
is purely advected, i.e. $\left( \partial_t + u\cdot\nabla \right) PV = 0$.
The second feature is that the velocity is determined by \emph{geostrophic balance} $u_h = \nabla_h^\perp p,$ which arises from the effect of the Coriolis force and the assumption of 
fast rotation.
In particular this means that the velocity has vanishing vertical component, and is thus purely horizontal, but still depends on all three spatial coordinates
(the velocity is \emph{not} two-dimensional). 
Here we have introduced the notation $v_h = (v_1,\, v_2, 0)$ to denote the horizontal components of any 3-vector $v,$
as well as the notation $w^\perp = (-w_2,\,w_1)$ to denote the $\frac{\pi}{2}$--rotation of any 2-vector $w$. Then we write $v_h^\perp = (-v_2,\,v_1,\,0)$.
The third feature is that the potential temperature is determined by \emph{hydrostatic balance} $\theta = \partial_3 p,$
which arises due to strong stratification.
Plugging these two balances into the definition of the potential vorticity tells us that
$
	PV = \Delta p,
$
and solving this elliptic PDE is known as \emph{$PV$ inversion}.

Noting that 
the velocity transporting the $PV$ 
may be reconstructed from the $PV$ itself via
$u = \nabla_h^\perp \Delta^{-1} PV,$ one can see that the QG system
is a nonlinear and nonlocal transport equation for the potential vorticity.
This is 
similar to how the vorticity form of the incompressible Euler equations is a nonlinear and nonlocal transport equation for the vorticity.

\textbf{The moist case.}
The question is now: ``What happens when moisture is incorporated into the model?''
In order to answer that question, we first recall that the moisture is modeled via a single scalar  unknown $q$ which describes the total water mixing ratio.
To be precise: $q$ is the mixing ratio measured as kilograms of total water per kilogram of dry air \cite{klein_majda_06, hernandez_duenas}.
For simplicity the saturation is taken to occur at $q=0$. This may be achieved without loss of generality by subtracting the mixing ratio at saturation, assumed here to be constant, from $q$.
We then introduced $H_u = \mathds{1} (q<0)$ and $H_s = 1-H_u = \mathds{1}(q\geqslant 0)$ to make it easier to see how 
physical relationships
vary depending on whether they occur in the unsaturated phase $ \left\{ q < 0 \right\}$ or in the saturated phase $ \left\{ q > 0 \right\}$.

While geostrophic balance remains the same once moisture is considered, 
hydrostatic balance only {\em appears} unaffected at first glance. Indeed, it is still true that $\partial_3 p = \theta$,
but it is now essential to split the potential temperature $\theta$ into two pieces, writing $\theta = \theta_e - q H_u$.
Here $\theta_e$ denotes the \emph{equivalent} potential temperature, and corresponds to the potential temperature that a parcel of air would have
if it were heated by converting its water vapour content into condensed water,
and $qH_u$ accounts for the latent heat stored in water vapour.
Once moisture is incorporated into the model, $\theta$ is no longer purely advected as was the case for dry dynamics, since source terms arise due to condensation and evaporation.
It is therefore $\theta_e$ which is purely advected, 
and is thus the variable used to define the moist potential vorticity $PV_e$. 
More precisely,  the moist potential vorticity is defined as $PV_e = {(\nabla\times u)}_3 + \partial_3 \theta_e$ (using $\theta_e$ and not $\theta$ as was done in the dry case).  

Physically speaking, it is important to introduce the concept of a slowly varying quantity whose dynamics is appropriate to describe temporal and spatial variations on synoptic scales.  In the atmospheric science literature, slowly varying quantities are also called ``balanced'' or simply ``slow.''
Mathematically, identification of appropriate slow variables follows from a distinguished limiting process of the governing equations, in this case  either the dry or moist rotating Boussinesq equations.
While dry $PV$ is slowly varying in the context of dry dynamics, it is moist $PV_e$ that is slowly varying in the presence of moisture and phase changes \cite{PQG2017}.  
By contrast, the variables $u$, $\theta_e$ and $q$ exhibit wave-like behavior, generally associated with smaller spatial scales and faster temporal scales. 
As slow variables in their respective dry and moist environments, $PV$ and $PV_e$ are of primary importance to meteorology because their dynamics are associated with synoptic-scale weather patterns.  
 
The description of moist synoptic-scale dynamics requires at least two slowly varying unknowns, instead of the single $PV$ variable in the dry case.
This is 
because 
there is a new unknown quantity $q$, without any new balances appearing (which would reduce the number of unknowns needed to fully describe the system).  Following from a distinguished limiting analysis of the two-phase moist Boussinesq equations, 
the new slow variable is denoted $M$ and defined as $M = \theta_e +q$ 
\cite{PQG2017}.

Note that we have, in the previous discussion and throughout the remainder of this paper, set a slew of physical constants to unity.
A ``Rosetta stone'' of sorts is provided in 
\fref{Appendix}{sec:rosetta}, which discusses the PQG model when these constants are present in their full glory.  \fref{Appendix}{sec:rosetta} acts as a dictionary between the notation of \cite{PQG2017} and this paper, 
and makes clear the fact that that there is no loss of generality in considering the simpler looking 
$PV_e$-and-$M$ inversion \eqref{eq:PV_M}.

\textbf{Derivation of $PV_e$-and-$M$ inversion.}
We are now ready to derive 
\eqref{eq:PV_M} and \eqref{eq:PV_M_min0}.
Using hydrostatic balance and the moist variable $M$ allows us to write the equivalent potential temperature $\theta_e$ as a function of $M$ and $\partial_3 p$ (see \fref{Corollary}{cor:q}).
Inserting that expression into the definition of the potential vorticity $PV_e$ produces the first form of $PV_e$-and-$M$ inversion recorded in \eqref{eq:PV_M}.

Alternatively, 
hydrostatic balance and the moist variable $M$ may also be used to write the water $q$ as a function of $M$ and $\partial_3 p$,
indicating that the sign of $q$ agrees everywhere with the sign of $M-\partial_3 p.$ 
This allows to rewrite $PV_e$-and-$M$ inversion as in \eqref{eq:PV_M_min0},
which is the formulation of $PV_e$-and-$M$ inversion used from now on.

\textbf{Brief summary of PQG literature.}
References
\cite{wetzel_smith_stechmann_19,wetzel_smith_stechmann_20} explain how $PV_e$-and-$M$ inversion can be used to extract the balanced component of a given atmospheric state, including the balanced parts of water, winds and temperature.  The authors emphasize that the $PV_e$-and-$M$ formulation of PQG is both balanced and invertible, whereas these key ingredients had not been combined together in previous studies of moist potential vorticity.  In order to achieve balance, the PQG equations were derived formally in \cite{PQG2017} as a distinguished asymptotic limit, and the limiting process has since garnered more attention from the analytical point of view in \cite{zhang_smith_stechmann_20_asymptotics}, and from the numerical point of view in \cite{zhang_smith_stechmann_20_numerics, zhang_smith_stechmann_22}.

Physical phenomena inherent in the PQG system have been explored through analytical and numerical studies aimed at specific classes of solutions that are important in meteorological applications \cite{wetzel_smith_stechmann_17, edwards_smith_stechmann_20, wetzel_smith_stechmann_19_fronts}.
When it comes to the study of the PQG equations ``at large'', i.e. away from specific classes of solutions, only numerical studies have been carried out so far
\cite{edwards_smith_stechmann_20_spectra, hu_edwards_smith_stechmann_21}.
It is worth noting that both of the numerical studies \cite{edwards_smith_stechmann_20_spectra, hu_edwards_smith_stechmann_21} 
consider the 2-level PQG system, in which the vertical dependence is simplified to include only two parallel horizontal slices.


\section{Variational formulation and basic properties of the energy}
\label{sec:var_form_and_prop_en}

In this section we begin working towards the proof of \fref{Theorems}{thm:wp}.
We begin by defining the energy, whose minimizers will be solutions of $\PV$-and-$M$ inversion, precisely.
We then study elementary properties of this energy pertaining to its convexity, coercivity, and differentiability.
These properties will then be used throughout the remainder of this paper.
Finally we conclude this section by proving that indeed minimizers of this energy are solutions of $\PV$-and-$M$ inversion.

Here is the energy in question.

\begin{definition}[Energy]
\label{def:energy}
	For any $M\in L^2$, $\PV\in H^{-1}$, and $p\in\mathring{H}^1$, where all spaces are taken over $\mathbb{T}^3$ and $H^{-1} \vcentcolon= {(\mathring{H}^1)}^*$, we define
	\begin{equation}
	\label{eq:def_en}
		E(p) \vcentcolon= \int_{\mathbb{T}^3} \frac{1}{2} {\lvert \nabla p \rvert}^2 - \frac{1}{4} {\min ( M - \partial_3 p,\, 0 )}^2 + \langle \PV,\, p \rangle.
	\end{equation}
\end{definition}

Here and throughout, unless specified otherwise, the pairing $ \langle \,\cdot\, , \,\cdot\, \rangle$ used in \eqref{eq:def_en} is the $H^{-1} \times \mathring{H}^1 $ duality pairing,
i.e. $ \langle \Lambda,\, p \rangle = \Lambda(p)$ for every $\Lambda\in H^{-1}$ and $p\in \mathring{H}^1 $.

As noted previously, an essential feature of the energy introduced in \fref{Definition}{def:energy} above is the lack of smoothness due to the presence of the minimum function.
As an intermediate step, it is therefore often useful to work with a smoothed out version of the problem. This is done by employing mollifiers.
To avoid any confusion with differing conventions we record below what we mean by a standard mollifier.

\begin{definition}[Mollifier]
\label{def:mollifier}
	A \emph{standard mollifier} is a non-negative smooth function $\varphi:\mathbb{R}\to\mathbb{R}$ whose support is contained in $(-1,\,1)$ and which is \emph{normalised},
	meaning that $\int \varphi = 1$. Moreover we say that $\varphi$ is \emph{centered} if $\int y \varphi (y) dy = 0$ -- i.e. the probability distribution $\varphi(y) dy$ has mean zero.
	We will often equivalently refer to $\varphi_\varepsilon$ as a mollifier to mean the family of mollifiers
	defined by $ \varphi_\varepsilon  \vcentcolon= \frac{1}{\varepsilon} \varphi \left( \frac{\cdot}{\varepsilon} \right)$.
\end{definition}

We now turn our attention to the study of the convexity of the energy.
First we show that the energy density is strongly convex -- see \fref{Lemma}{lemma:equiv_charac_strg_conv} for a reminder of equivalent characterisations of strong convexity.

\begin{lemma}[Strong convexity of the energy density]
\label{lemma:strg_cnvx_e}
	For any fixed $r\in\mathbb{R}$ the function $e:\mathbb{R}^3\to\mathbb{R}$ defined as $e(u) \vcentcolon= \frac{1}{2} {\lvert u \rvert}^2 - \frac{1}{4} \min {\left( r - u_3,\, 0 \right)}^2$
	is $ \frac{1}{2} $-strongly convex.
\end{lemma}
\begin{proof}
	Write $e(u) = \frac{1}{2} {\lvert u \rvert}^2 - \frac{1}{2} f_0 (r - u_3)$ for $f_0(s) \vcentcolon= \frac{1}{2} {\min(s,\,0)}^2$.
	For $\varphi_\varepsilon$ a \hyperref[def:mollifier]{standard mollifier} define $f_\varepsilon \vcentcolon= f_0 * \varphi_\varepsilon$
	and $e_\varepsilon (u) \vcentcolon= \frac{1}{2}  {\lvert u \rvert}^2 - \frac{1}{2}  f_\varepsilon ( r-u_3)$.
	Then we have that
	$
		\nabla^2 e_\varepsilon (u) = I - \frac{1}{2} f_\varepsilon '' (r - u_3) e_3 \otimes e_3
	$
	where $ \norm{f_\varepsilon''}{L^\infty} \leqslant \text{Lip} (f_0') = 1$ since $f_0'(s) = \min(s,\,0)$,
	and hence $\nabla^2 e_\varepsilon \geqslant \frac{1}{2} I$ which shows, by virtue of \fref{Lemma}{lemma:second_order_charac_strg_conv}
	that $e_\varepsilon$ is $ \frac{1}{2} $-strongly convex.
	We may then pass to the limit as $\varepsilon\to 0$ in any of the \emph{first-order} characterisations of strong convexity of \fref{Lemma}{lemma:equiv_charac_strg_conv} to  deduce
	(since $f_0 \in C^1$, and hence $e\in C^1$ such that $e_\varepsilon \to g$ in $C^1$) that $e$ itself is $\frac{1}{2}$-strongly convex.
\end{proof}

We continue studying the convexity of the energy, now proving that the energy itself is strongly convex.
The weak lower semi-continuity of the energy then follows.
We also record a coercivity estimate for the energy.

\begin{prop}[Properties of the energy]
\label{prop:prop_en}
	The energy $E$ introduced in \fref{Definition}{def:energy} is $\frac{1}{2}$-strongly convex and weakly lower semi-continuous over $\mathring{H}^1$.
	Moreover $E$ is coercive over $\mathring{H}^1$ in the sense that 
	$
		E(p) \geqslant \frac{1}{16}\norm{p}{ \mathring{H}^1 }^2 - \frac{3}{4} \norm{M}{L^2}^2 - 4 \norm{\PV}{H^{-1}}^2
	$
	for every $p\in \mathring{H}^1 $.
\end{prop}
\begin{proof}
	We may write
	\begin{equation}
	\label{eq:write_E_using_e_M}
		E(p) = \int_{\mathbb{T}^3} e_M \left( x,\, \nabla p(x) \right) dx + \langle \PV,\,p \rangle
	\end{equation}
	for $e_M(x,\,u) \vcentcolon= \frac{1}{2} {\lvert u \rvert}^2 - \frac{1}{4} {\min \left( M(x) - u_3,\, 0 \right)}^2$.
	\fref{Lemma}{lemma:strg_cnvx_e} then tells us that, for any $M\in L^2$, $e_M (x, \,\cdot\,)$ is $ \frac{1}{2} $-strongly convex for almost every $x$ in $\mathbb{T}^3$,
	and so the strong convexity of $E$ follows.
	The weak lower semi-continuity then follows from \fref{Lemma}{lemma:cvx_en_implies_wk_lsc} since strong convexity implies convexity.
	We now turn our attention to the coercivity of $E$.
	We introduce $H_u \vcentcolon= \mathds{1} (M < \partial_3 p)$ and $H_s \vcentcolon= 1-H_u = \mathds{1} (M \geqslant \partial_3 p)$ to rearrange $E$ such that
	\begin{align}
		E(p) - \langle \PV,\,p \rangle
	\label{eq:prop_en_int}
		= \int_{\mathbb{T}^3} \frac{1}{2} {\lvert \nabla p \rvert}^2  H_s + \Bigg(
			\frac{1}{2} {\lvert \nabla_h p \rvert}^2 + \underbrace{
				\frac{1}{4} {(\partial_3 p)}^2 + \frac{1}{2} M \partial_3 p - \frac{1}{4} M^2
			}_{ =\vcentcolon (\star) }
		\Bigg) H_u.
	\end{align}
	We note that, by Cauchy-Schwarz, $(\star) \geqslant \frac{1}{8} {(\partial_3 p)}^2 - \frac{3}{4} M^2$.
	Combining this inequality with \eqref{eq:prop_en_int} above we see that $E(p) - \langle \PV,\, p \rangle \geqslant \frac{1}{8} \norm{p}{ \mathring{H}^1 }^2 - \frac{3}{4} \norm{M}{L^2}$.
	To conclude we apply Cauchy-Schwarz to $ \langle \PV,\, p \rangle$.
\end{proof}

\begin{remark}[Non-negativity of the energy]
\label{rmk:pos_en}
	\fref{Proposition}{prop:prop_en} shows that the quadratic part of $E$, namely $E - \PV$, may not necessarily be non-negative.
	Indeed, for $M = -\partial_3 p$ and $\nabla_h p = 0$ we read from \eqref{eq:prop_en_int} that $E(p) - \langle \PV,\, p \rangle = \int_{\mathbb{T}^3} \frac{1}{2} M^2 (H_s - H_u)$.
	If one sought to work with an energy whose quadratic part was non-negative this issue would be easily remedied by adding $ \frac{1}{2}  \int_{\mathbb{T}^3} M^2$ to the energy:
	this has no bearing on the minimization of the energy with respect to $p$ and a Cauchy-Schwarz argument as in the proof of \fref{Proposition}{prop:prop_en}
	shows that then $E - \PV + \frac{1}{2} \norm{M}{L^2} \geqslant 0$.
	(This is actually optimal: the example above where $M = - \partial_3 p$ shows that adding any smaller multiple of the $L^2$ norm of $M$ to the energy would not
	be sufficient to guarantee the non-negativity of its quadratic part.)
\end{remark}

Having studied the convexity and coercivity of the energy we now turn our attention towards its differentiability.
First we show that the energy is G\^{a}teaux differentiable and obtain an expression for its derivative.

\begin{lemma}[G\^{a}teaux differentiability of the energy]
\label{lemma:Gateaux_diff_energy}
	The energy $E$ introduced in \fref{Definition}{def:energy} is G\^{a}teaux differentiable on $\mathring{H}^1$ with G\^{a}teaux derivative at any $p\in\mathring{H}^1$,
	denoted by $DE(p)$, given by
	\begin{equation}
	\label{eq:Gateaux_diff_en}
		DE(p) \phi = \int_{\mathbb{T}^3} \nabla p \cdot \nabla\phi + \frac{1}{2} \min ( M-\partial_3 p,\, 0) \partial_3 \phi
			+ \langle \PV,\, \phi \rangle \text{ for every } \phi\in\mathring{H}^1.
	\end{equation}
\end{lemma}
\begin{proof}
	Since $\PV\in H^{-1}$ we know that $p\mapsto \langle \PV,\, p \rangle$ is G\^{a}teaux differentiable (and equal to its derivative)
	and so we focus our attention on the remaining terms in the energy, namely
	$
		E_0 (p) \vcentcolon= \int_{\mathbb{T}^3} \frac{1}{2} {\lvert \nabla p \rvert}^2 - \frac{1}{4} {\min (M - \partial_3 p,\, 0)}^2.
	$
	It suffices to show that $E_0$ is G\^{a}teaux differentiable with
	\begin{equation}
	\label{eq:Gat_diff_int_1}
		DE_0 (p) \phi
		= \int_{\mathbb{T}^3} \nabla p\cdot\nabla \phi + \frac{1}{2} \min ( M - \partial_3 p,\, 0) \partial_3 \phi
		\text{ for every } \phi \in \mathring{H}^1.
	\end{equation}
	To do that we fix $p,\,\phi\in\mathring{H}^1$ and define
	\begin{equation*}
		\tilde{e}(t,\,x) \vcentcolon= \frac{1}{2} {\lvert \nabla p (x) + t \nabla\phi (x) \rvert}^2 
		- \frac{1}{4} {\min \left( M(x) - \left[ \partial_3 p (x) + t \partial_3 \phi (x) \right] ,\, 0 \right)}^2
	\end{equation*}
	such that
	$
		E_0 (p + t\phi) = \int_{\mathbb{T}^3} \tilde{e}(t,\,x) dx.
	$
	Since $\tilde{e}(t,\,\cdot\,)\in L^1(\mathbb{T}^3)$ for all $t$ with
	\begin{equation}
	\label{eq:Gat_diff_int_2}
		\frac{\partial \tilde{e}}{\partial t} (t,\,x)
		= \left[ \nabla p (x) + t\nabla \phi (x) \right] \cdot \nabla \phi (x)
		+ \frac{1}{2} \min \left( M(x) - \left[ \partial_3 p (x) + t\partial_3 \phi (x) \right],\, 0 \right) \partial_3 \phi (x)
	\end{equation}
	an application of the Dominated Convergence Theorem (more precisely Theorem 2.27 in \cite{folland}) tells us that $t \mapsto E_0 (p+t\phi)$ is differentiable, with
	$
	\frac{\mathrm{d}}{\mathrm{d}t} E_0 (t + t\phi) = \int_{\mathbb{T}^3} \frac{\partial \tilde{e}}{\partial t}.
	$
	Evaluating at $t=0$ and using \eqref{eq:Gat_diff_int_2} then yields \eqref{eq:Gat_diff_int_1}.
\end{proof}

We now bootstrap up from G\^{a}teaux to Fr\'{e}chet differentiability in light of the continuity of the G\^{a}teaux derivative of the energy.

\begin{cor}[Fr\'{e}chet differentiability of the energy]
\label{cor:Frechet_diff_en}
	The G\^{a}teaux derivative of the energy $E$ introduced in \fref{Definition}{def:energy} is Lipschitz, with specifically the estimate
	$
		\norm{ DE(p_1) - DE(p_2) }{H^{-1}} \leqslant \frac{3}{2} \norm{p_1 - p_2}{\mathring{H}^1}
	$
	for every $p_1,\,p_2\in \mathring{H}^1 $,
	and so $E$ is Fr\'{e}chet differentiable.
\end{cor}
\begin{proof}
	This follows from the expressions for $DE$ recorded in \fref{Lemma}{lemma:Gateaux_diff_energy} and a simple estimate using the fact that the map $s\mapsto \min (s,\,0)$ is $1$-Lipschitz.
\end{proof}

With the Fr\'{e}chet differentiability of the energy in hand we may now obtain a simple (but very handy!) estimate that shows that
as a consequence of its strong convexity, the energy essentially acts as a norm about its minimizer.

\begin{lemma}[The energy is essentially a norm abouts its minimizer]
\label{lemma:en_is_norm_near_min}
	Suppose that the energy $E$ introduced in \fref{Definition}{def:energy} has a minimizer $p^* \in \mathring{H}^1 $.
	For any $p\in \mathring{H}^1$ the following estimate holds:
	$
		\norm{p - p^*}{\mathring{H}^1}^2
		\leqslant 4 \left( E(p) - E(p^*) \right)
		= 4 \left( E(p) - \min E \right).
	$
\end{lemma}
\begin{proof}
	\fref{Corollary}{cor:Frechet_diff_en} tells us that $E$ is Fr\'{e}chet differentiable and so the result follows from
	combining \fref{Proposition}{prop:prop_en} and \fref{Lemma}{lemma:equiv_charac_strg_conv}.
	Since $E$ is $ \frac{1}{2} $-strongly convex and $DE(p^*) = 0$ at the minimizer we obtain that, for any $p\in\mathring{H}^1$,
	\begin{equation*}
		E(p) - E(p^*)
		\geqslant \langle DE(p^*),\, p-p^* \rangle + \frac{1/2}{2} \norm{p - p^*}{\mathring{H}^1}^2
		= \frac{1}{4} \norm{p-p^*}{\mathring{H}^1}^2,
	\end{equation*}
	as desired.
\end{proof}

We now conclude this section by proving what is perhaps its main result,
namely that minimizers of the energy introduced in \fref{Definition}{def:energy} are precisely the solutions of $\PV$-and-$M$ inversion.

\begin{lemma}[Variational formulation of $\PV$-and-$M$ inversion]
\label{lemma:var_form_PV_and_M_inv}
	For any $p \in \mathring{H}^1 $, $p$ is a global minimizer of $E$ introduced in \fref{Definition}{def:energy} if and only if
	it is an $ \mathring{H}^1 $-weak solution of
	\begin{equation}
	\label{eq:PV_M_inv_strong}
		- \Delta p - \frac{1}{2} \partial_3 \left( \min ( M-\partial_3 p,\,\ 0) \right) = -\PV
	\end{equation}
	in the sense that
	\begin{equation}
	\label{eq:PV_M_inv_weak}
		\int_{\mathbb{T}^3} \nabla p \cdot \nabla \phi + \frac{1}{2} \min (M - \partial_3 p,\, 0) (\partial_3 \phi)
		= - \langle \PV,\, \phi \rangle
		\text{ for every } \phi\in \mathring{H}^1 .
	\end{equation}
\end{lemma}
\begin{proof}
	\fref{Lemma}{lemma:Gateaux_diff_energy} tells us that \eqref{eq:PV_M_inv_weak} is equivalent to $DE(p) = 0$.
	Since $E$ is differentiable and strictly convex (c.f. \fref{Corollary}{cor:Frechet_diff_en} and \fref{Proposition}{prop:prop_en}) the claim then follows immediately.
\end{proof}


\section{Well-posedness}
\label{sec:wp}

In this section we use the variational formulation of $\PV$-and-$M$ inversion to show that the problem is well-posed.
We first show that unique solutions exist in \fref{Theorem}{thm:exist_and_unique} and then turn our attention to their regularity.

As discussed previously, the de Giorgi theory is central to that matter of regularity and so
the first main result of this section is \fref{Theorem}{thm:Holder_cty_grad} in which the H\"{o}lder regularity of the gradient of solutions of $\PV$-and-$M$ inversion is established.
With this key result in hand we then turn our attention to the matter of higher regularity, establishing sufficient conditions on $\PV$ and $M$ for the solutions to be classical, in each phase,
in \fref{Theorem}{thm:classical_sols_each_phase} and finally establishing sufficient conditions on $\PV$ and $M$ for the solutions to be smooth in each phase in \fref{Theorem}{thm:higher_reg}.

We conclude this section with a discussion of how to construct explicit one-dimensional ($x_3$--dependent) solutions, which in particular show that the regularity theory developed here is nearly sharp.

First, as promised, we establish existence and uniqueness.

\begin{theorem}[Existence and uniqueness]
\label{thm:exist_and_unique}
	For every $M\in L^2$, and $\PV\in H^{-1}$ the energy $E$ introduced in \fref{Definition}{def:energy}
	has a unique global minimizer in $\mathring{H}^1$ which is also the unique weak solution of \eqref{eq:PV_M_inv_weak}.
\end{theorem}
\begin{proof}
	For any $p\in \mathring{H}^1$, $M\in L^2$, and $\PV \in H^{-1}$ the energy $E(p)$ is finite
	and the coercivity of the energy recorded in \fref{Proposition}{prop:prop_en} tells us that $E(p)$ is bounded below independently of $p$.
	Therefore $\inf_{\mathring{H}^1} E$ is well-defined.
	We may thus take a minimizing sequence $(p_n ) \subseteq \mathring{H}^1$ satisfying $E(p_n) \downarrow \inf E$ as $n\to\infty$.
	The coercivity of $E$ implies that the sequence $(p_n)$ is bounded in $\mathring{H}^1$,
	and so we deduce from Kakutani's theorem that there is a subsequence $(p_{n_k})$ which is $\mathring{H}^1$-weakly convergent to some $p_\infty \in \mathring{H^1}$.
	Finally, since \fref{Proposition}{prop:prop_en} tells us that the energy is weakly lower semi-continuous in $\mathring{H}^1$, we deduce that
	$
		E(p_\infty) \leqslant \liminf_{k\to\infty} E(p_{n_k}) = \inf E,
	$
	i.e. $p_\infty$ is a minimizer of $E$, as desired.
	In particular the strict convexity of $E$ established in \fref{Proposition}{prop:prop_en} guarantees that this minimizer is unique.
	\fref{Lemma}{lemma:var_form_PV_and_M_inv} then ensures that this unique minimizer is precisely the unique weak solution of \eqref{eq:PV_M_inv_weak}.
\end{proof}

We now turn our attention to the regularity of solutions of $\PV$-and-$M$ inversion.
First we obtain that if $\PV$ and $M$ are sufficiently regular then the solutions lie in $H^2$.
This argument relies on the use of finite differences (see \fref{Definition}{def:fin_diff} and \fref{Lemma}{lemma:fin_diff}).

\begin{lemma}[$H^2$ regularity]
\label{lemma:H_2_reg}
	Let $M\in H^1$ and $\PV\in L^2$.
	Any $\mathring{H}^1$-weak solution $p$ of \eqref{eq:PV_M_inv_weak} belongs to $H^2$ and satisfies
	$
		\norm{p}{\dot{H}^2} \leqslant \norm{M}{\dot{H}^1} + 2 \norm{\PV}{L^2}.
	$
\end{lemma}
\begin{proof}
	We define, for $g_0(s) \vcentcolon= \min(s,\,0)$, and for any $x\in\mathbb{T}^3$, the function
	$
		\zeta_h^i (x) \vcentcolon= \int_0^1 g_0' \left( (I + \theta\Delta_h^i)(M-\partial_3 p) (x) \right) d\theta
	$
	where recall that the finite difference operator $\Delta_h^i$ is introduced in \fref{Definition}{def:fin_diff},
	and note that this function is well-defined almost everywhere independently of the representative of $M-\partial_3 p \in L^2$ used since
	for any fixed $h$ and $i$, defining $\zeta_h^i$ only requires the evaluation of $M-\partial_3 p$ and $(M - \partial_3 p)(\,\cdot\, + he_i)$.
	Note that $\zeta$ is defined to be precisely the term that will appear, by virtue of the chain rule for finite differences recorded in \fref{Lemma}{lemma:fin_diff},
	when applying a finite difference to the function
	$
		\min(M~-~\partial_3 p,\,0)~=~g_0 (M~-~\partial_3 p).
	$
	Note also that, since the Lipschitz constant of $g_0$ is equal to one, $ \lvert \zeta_h^i (x) \rvert \leqslant 1$ for every $x\in\mathbb{T}^3$.

	Now we use $\Delta_{-h}^i \Delta_h^i p$ as a test function (i.e. ${(\Delta_h^i )}^* \Delta_h^i$, as indicated by the computation of the adjoint
	of the finite difference operator in \fref{Lemma}{lemma:fin_diff}).
	This produces, using \fref{Lemma}{lemma:fin_diff},
	\begin{equation*}
		\int (\PV) \Delta_{-h}^i \Delta_h^i p
		= \int {\lvert \nabla\Delta_h^i p \rvert}^2 - \frac{1}{2} {(\partial_3 \Delta_h^i p)}^2\zeta_h^i + \frac{1}{2} (\Delta_h^i M) (\partial_3 \Delta_h^i p) \zeta_h^i.
	\end{equation*}
	Therefore, since $ \lvert \zeta_h^i \rvert \leqslant 1$,
	\begin{align*}
		\frac{1}{2} \int {\lvert \nabla\Delta_h^i p \rvert}^2
		&\leqslant \int {\lvert \nabla\Delta_h^i p \rvert}^2 - \frac{1}{2} {(\partial_3 \Delta_h^i p)}^2 \zeta_h^i
	\\
		&= \int- \frac{1}{2} (\Delta_h^i M)(\partial_3 \Delta_h^i p) \zeta_h^i + \int (\Delta_h^i \PV)(\Delta_h^i p)
	\\
		&\leqslant \left(
			\frac{1}{2} \norm{\Delta_h^i M}{L^2} + \norm{\Delta_h^i \PV}{H^{-1}} 
		\right) \norm{\nabla\Delta_h^i p}{L^2} 
	\end{align*}
	such that, using \fref{Lemma}{lemma:fin_diff} once more,
	\begin{equation*}
		\norm{\Delta_h^i \nabla p}{L^2}
		\leqslant \norm{\Delta_h^i M}{L^2} + 2 \norm{\Delta_h^i \PV}{H^{-1}} 
		\leqslant \lvert h \rvert \left( \norm{\nabla M}{L^2} + 2 \norm{\PV}{L^2} \right).
	\end{equation*}
	So finally we use \fref{Lemma}{lemma:fin_diff} one last time to conclude that $\nabla^2 p \in L^2$, with the estimate
	$
		\norm{p}{\dot{H}^2} = \norm{\nabla^2 p}{L^2} \leqslant \norm{M}{\dot{H}^1} + 2 \norm{\PV}{L^2}.
	$
\end{proof}

The next step in climbing the regularity ladder is to prove that the gradient of solutions to $\PV$-and-$M$ inversion is H\"{o}lder continuous.
Following de Giorgi, we do this by considering the equation satisfied by the gradient of such solutions.
That equation is recorded in the result below.

\begin{lemma}[Equation satisfied by the gradient]
\label{lemma:eqtn_satisfied_by_gradient}
	Let $M\in H^1$ and $\PV\in L^2$.
	If $p\in\mathring{H}^1$ is an $\mathring{H}^1$-weak solution of \eqref{eq:PV_M_inv_strong} then $u_i \vcentcolon= \partial_i p$ is an $\mathring{H}^1$-weak solution of
	\begin{equation}
	\label{eq:eqtn_satisfied_grad}
		-\Delta_h u_i
		- \frac{1}{2} \partial_3 \Big( [1 + \mathds{1} (M \geqslant \partial_3 p)] (\partial_3 u_i)\Big)
		= - \partial_i \PV + \frac{1}{2} \partial_3 \Big( \mathds{1} (M < \partial_3 p) (\partial_i M) \Big)
	\end{equation}
	in the sense that, for every $\phi\in\mathring{H}^1$,
	\begin{align*}
		\int_{\mathbb{T}^3} \nabla_h u_i \cdot \nabla_h \phi
		+ \frac{1}{2} [ 1 + \mathds{1} (M \geqslant \partial_3 p) ] (\partial_3 u_i) (\partial_3 \phi)
	\\
		= - \langle \partial_i \PV,\, \phi \rangle
		- \int_{\mathbb{T}^3} \frac{1}{2} \mathds{1} (M<\partial_3 p)(\partial_i M) (\partial_3 \phi).
	\end{align*}
\end{lemma}
\begin{proof}
	As usual we write $g_0(s) \vcentcolon= \min(0,\,s)$.
	Since $M\in H^1$ and $ \PV \in L^2$ \fref{Lemma}{lemma:H_2_reg} tells us that $p\in H^2$, and so $M - \partial_3 p \in H^1$.
	Since $g_0$ is Lipschitz and $g_0(0) = 0$ we deduce that $g_0 (M-\partial_3 p)$ belongs to $H^1$ (see Theorem 12.69 of \cite{leoni}),
	with
	$
		\partial_i \left( g_0 ( M -\partial_3 p) \right)
		= g_0' (M - \partial_3 p) ( \partial_i M - \partial_3 \partial_i p)
		= \mathds{1} ( M < \partial_3 p) ( \partial_i M - \partial_3 \partial_i p).
	$
	In particular, for any $\psi\in \mathring{C}^\infty$ (the space of smooth functions with average zero) we may use $ \partial_i \psi$ as a test function to see that,
	using the weak formulation recorded in \fref{Lemma}{lemma:var_form_PV_and_M_inv} and integrating by parts,
	\begin{align*}
		- \langle \partial_i \PV,\, \psi \rangle
		= { \left( \PV,\, \partial_i \psi \right) }_{L^2} 
		= \int_{\mathbb{T}^3} \nabla \partial_i p \cdot \nabla \psi + \frac{1}{2} g_0' (M-\partial_3 p) ( \partial_i M - \partial_3 \partial_i p) ( \partial_3 \psi)
	\\
		= \int_{\mathbb{T}^3} \nabla_h u_i \cdot \nabla_h \psi
			+ \left( 1 - \frac{1}{2} \mathds{1} (M < \partial_3 p) \right) ( \partial_3 u_i ) ( \partial_3 \psi)
			+ \frac{1}{2} \mathds{1} (M < \partial_3 p) ( \partial_i M ) ( \partial_3 \psi)
	\end{align*}
	and so indeed
	\begin{align}
		&\int_{\mathbb{T}^3} \nabla_h u_i \cdot \nabla_h \psi
		+ \frac{1}{2} \left( 1 + \mathds{1} (M \geqslant \partial_3 p) \right) ( \partial_3 u_i ) ( \partial_3 \psi)
	\nonumber\\
		&= \int_{\mathbb{T}^3} - \frac{1}{2} \mathds{1} (M < \partial_3 p) ( \partial_i M ) ( \partial_3 \psi)
		- \langle \partial_i \PV,\, \psi \rangle.
	\label{eq:eqtn_satisfied_by_gradient_int}
	\end{align}
	In particular, since $u_i = \partial_i p \in H^1$, $ \partial_i M \in L^2$, and $ \partial_i \PV \in H^{-1}$ we may for any $\phi\in\mathring{H}^1$
	find an approximating sequence ${(\psi_n)}_n \subseteq \mathring{C}^\infty$ converging to $\psi$ in $\mathring{H}^1$ such that \eqref{eq:eqtn_satisfied_by_gradient_int}
	holds for $\psi = \psi_n$ and hence, passing to the limit, \eqref{eq:eqtn_satisfied_by_gradient_int} also holds for $\psi = \phi$, as desired.
\end{proof}

We are now equipped to prove the first main result of this section, which states that the gradient of solutions of $\PV$-and-$M$ inversion is H\"{o}lder continuous
provided that $\PV$ and $M$ are sufficiently regular.
The proof of this result makes uses of classical ideas in the sense that it relies on framing the de Giorgi argument in a manner compatible with the Campanato formulation of H\"{o}lder continuity
(see \fref{Definition}{def:Campanato} and \fref{Lemma}{lemma:Campanato_implies_Holder} ).
We include the proof here for the following reason:
a plethora of expository work (of which we only mention \cite{cafarelli_vasseur, vasseur}) discusses the \emph{homogeneous} de Giorgi method
in the elliptic case where there is no forcing term
and some textbooks such \cite{han_lin} discuss the \emph{inhomogeneous} case by analogy with the classical treatment of elliptic equations with forcing,
but an explicit proof, both sufficiently detailed and simple for the non-specialist, is difficult to find. The intent is that the proof provided here fits that bill.
(Note that auxiliary results used in the proof are relegated to \fref{Appendix}{sec:tools_reg_th} -- we emphasize the assembly of these ideas here.)

\begin{theorem}[H\"{o}lder continuity of the gradient]
\label{thm:Holder_cty_grad}
	Suppose that $M\in H^1$ satisfies $\nabla M\in L^q ( \left\{ M < \partial_3 p  \right\} )$ and that $\PV\in L^q$ for some $q > d = 3$.
	There exists $\alpha = \alpha(q) \in \left( 0,\, 1 - d/q \right]$ such that if $p$ is an $\mathring{H}^1$-weak solution of \eqref{eq:PV_M_inv_strong}
	then $\nabla p \in C^{0,\,\alpha}$ and the following estimate holds:
	\begin{equation}
	\label{eq:Holder_cty_grad_est}
		\norm{\nabla p}{C^{0,\,\alpha} (\mathbb{T}^3)} \leqslant C \left(
			\norm{\nabla M}{L^q ( M < \partial_3 p )} + \norm{\PV}{L^q (\mathbb{T}^3 )} 
		\right)
	\end{equation}
	where $C = C(q) > 0$.
\end{theorem}
\begin{proof}
	Fix $i = 1,\,2,\,3$ and write $u \vcentcolon= \partial_i p$.
	Note that, as derived in \fref{Lemma}{lemma:eqtn_satisfied_by_gradient}, $u$ is an $\mathring{H}^1$-weak solution of \eqref{eq:eqtn_satisfied_grad}
	which may be written as
	\begin{equation}
	\label{eq:Holder_cty_grad_eq}
		-\nabla\cdot ( A(x)\nabla u ) = \nabla\cdot F
	\end{equation}
	for
	$
		A \vcentcolon= \widetilde{I} + \frac{1}{2} \left( 1 + \mathds{1} (\partial_3 p < M) \right) e_3\otimes e_3
	$ and $
		F \vcentcolon= \PV e_i + \frac{1}{2} \mathds{1} ( M < \partial_3 p ) ( \partial_i M ) e_3,
	$
	where $\widetilde{I} = I - e_3\otimes e_3$.
	Crucially: $A\in L^\infty$ is uniformly elliptic since
	$
		A(x)\xi\cdot\xi \geqslant \frac{1}{2} {\lvert \xi \rvert}^2
	$
	for all $x\in\mathbb{R}^3$.

	We will now use \eqref{eq:Holder_cty_grad_eq} to prove that, for any $x_0\in\mathbb{T}^3$ and $r>0$,
	\begin{equation}
	\label{eq:Holder_cty_grad_Campanato}
		\int_{B(x_0,\,r)} {\lvert \nabla u \rvert}^2 \leqslant C \norm{F}{L^q}^2 r^{d-2+2\alpha}
	\end{equation}
	for some $\alpha = \alpha(q) \in (0,\,1)$ and $C = C(q) > 0$ (recall that $d=3$).
	It will then follow from \fref{Proposition}{prop:suff_Campanato_cond_gradients}
	that $u\in C^{0,\,\alpha}$, with the desired estimate \eqref{eq:Holder_cty_grad_est}.

	So let us fix $x_0\in\mathbb{T}^3$ and $r>0$.
	To establish \eqref{eq:Holder_cty_grad_Campanato} we split the solution $u$, inside the ball $B = B(x_0,\,r)$, into two parts.
	One part, $v$, accounts for the forcing term $\nabla\cdot F$ while the other part, $w$, solves a homogeneous problem (i.e. without forcing)
	in order to correct for the homogeneous Dirichlet boundary conditions imposed on $v$.
	More precisely we define $v$ to be the $\mathring{H}^1$-weak solution of
	\begin{equation*}
		\left\{
		\begin{aligned}
			&-\nabla\cdot ( A(x) \nabla v ) = \nabla\cdot F 	&\text{ in } B \text{ and }\\
			&v=0							&\text{ on } \partial B,
		\end{aligned}
		\right.
	\end{equation*}
	i.e. $v\in H_0^1 (B)$, and define $w \vcentcolon= u - v$, such that indeed $u = v + w$, which means that $w$ is an $H^1$-weak solution of
	$
		-\nabla\cdot (A(x) \nabla w) = 0
	$
	in $B$
	(but we know nothing, and need to know nothing, about the boundary values of $w$).
	Combining \fref{Propositions}{prop:Campanato_formulation_hom_DG} and \ref{prop:Campanato_type_a_priori_estimate_div}
	we deduce that there exist $\alpha_1$ and $\alpha_2 = 1-\frac{d}{q}$ both in $(0,\,1)$,
	such that, for any $0 < \rho \leqslant r$,
	\begin{align*}
		\int_{B(x_0,\,\rho)} {\lvert \nabla u \rvert}^2 
		&\leqslant 2 \left(
			\int_{B(x_0,\,\rho)} {\lvert \nabla w \rvert}^2 
			+ \int_{B(x_0,\,\rho)} {\lvert \nabla v \rvert}^2 
		\right)
	\\
		&\leqslant C \left(
			{ \left( \frac{\rho}{r} \right) }^{d-2+2\alpha_1} \int_{B(x_0,\,r)} {\lvert \nabla w \rvert}^2 
			+ \int_{B(x_0,\,r)} {\lvert \nabla v \rvert}^2 
		\right)
	\\
		&\leqslant C \left(
			{ \left( \frac{\rho}{r} \right) }^{d-2+2\alpha_1} \int_{B(x_0,\,r)} {\lvert \nabla u \rvert}^2 
			+ r^{d-2+2\alpha_2} \norm{F}{L^q}^2
		\right).
	\end{align*}
	We are now in a position to make use of \fref{Lemma}{lemma:technical_lemma_inhom_DG} with
	$\phi(\rho) = \int_{B(x_0,\,\rho)} {\lvert \nabla u \rvert}^2$,
	$B = C \norm{F}{L^q}^2$,
	$\beta = d-2+2\min(\alpha_1,\,\alpha_2)$, and 
	$\gamma = d-2+2\alpha_1$.
	We deduce that there exists $\delta \in (\beta,\,\gamma)$ such that, for $\alpha = \min(\alpha_1,\,\alpha_2)$, and since $0 < \rho \leqslant R_0 < 1$,
	\begin{align}
		\int_{B(x_0,\,\rho)} {\lvert \nabla u \rvert}^2 
		\leqslant C \left[
			{ \left( \frac{\rho}{R_0/2} \right) }^\delta \phi \left( \frac{R_0}{2} \right) + B \rho^\beta
		\right]
		\leqslant C \left[ \phi \left( \frac{R_0}{2} \right) + B \right] \rho^\beta
	\nonumber\\
		\leqslant C \left( \int_{\mathbb{T}^3} {\lvert \nabla u \rvert}^2 + \norm{F}{L^q}^2 \right) \rho^{d-2+2\alpha}.
	\label{eq:Holder_cty_grad_Campanato_int}
	\end{align}
	In particular note that $\alpha \leqslant \alpha_2 = 1 - \frac{d}{q}$.
	To conclude it suffices to use \fref{Lemma}{lemma:std_H_1_est_unif_ell_op} and H\"{o}lder's inequality, such that
	$
		\int_{\mathbb{T}^3} {\lvert \nabla u \rvert}^2 \leqslant C (q) \norm{F}{L^q}^2,
	$
	and hence obtain \eqref{eq:Holder_cty_grad_Campanato} from \eqref{eq:Holder_cty_grad_Campanato_int}.
	\fref{Proposition}{prop:suff_Campanato_cond_gradients} then tells us that
	$
		\norm{u}{C^{0,\,\alpha}} \leqslant C(q) ( \norm{F}{L^q} + \norm{u}{L^2} )
	$
	and so we may use \fref{Lemma}{lemma:std_H_1_est_unif_ell_op} once again to finally obtain \eqref{eq:Holder_cty_grad_est}.
\end{proof}

As is often the case in elliptic regularity theory, having established that the gradient of solutions is H\"{o}lder continuous we immediately deduce useful consequences
and higher regularity.
First, a useful consequence: if $\PV$ and $M$ are sufficiently regular then both the unsaturated phase $ \left\{ M < \partial_3 p \right\}$ and the saturated phase $ \left\{ \partial_3 p < M \right\}$
are open sets.

\begin{cor}[Each phase is an open set]
\label{cor:phases_are_open}
	Let $M\in H^1$ be continuous with $\nabla M \in L^q ( M < \partial_3 p )$ and let $\PV\in L^q$ for some $q > d = 3$.
	If $p$ is an $\mathring{H}^1$-weak solution of \eqref{eq:PV_M_inv_strong} then $\partial_3 p - M$ is continuous and so both the
	unsaturated phase $ \left\{ M < \partial_3 p \right\}$ and the saturated phase $ \left\{ \partial_3 p < M \right\}$ are open sets.
\end{cor}
\begin{proof}
	This follows immediately from \fref{Theorem}{thm:Holder_cty_grad} which tells us that, under these assumptions on $M$ and $\PV$,
	$\nabla p$ is H\"{o}lder continuous.
\end{proof}

We now continue deducing further results from the H\"{o}lder continuity of the gradient obtained in \fref{Theorem}{thm:Holder_cty_grad} above.
Here we obtain higher-order regularity, namely sufficient conditions for the H\"{o}lder continuity of the \emph{Hessian} of solutions of $\PV$-and-$M$ inversion.

\begin{lemma}[H\"{o}lder continuity of the Hessian away from the interface]
\label{lemma:Holder_cty_Hessian}
	Let $M\in H^1$ be continuous with both $\nabla M$ and $\partial_3 \nabla M$ belonging to $L^q \left( M < \partial_3 p \right)$,
	let $\PV \in L^q$, and let $\nabla \PV \in L^q \left( M \neq \partial_3 p \right)$ for some $q > d = 3$.
	If $p$ is an $\mathring{H}^1$-weak solution of \eqref{eq:PV_M_inv_strong} then $\nabla^2 p \in C^{0,\,\alpha}_\text{loc} \left( M \neq \partial_3 p \right)$
	for $\alpha = 1 - \frac{d}{q}$.
\end{lemma}
\begin{proof}
	\fref{Lemma}{lemma:eqtn_satisfied_by_gradient} tells us that $u = \partial_i p$ is an $\mathring{H}^1$-weak solution of \eqref{eq:eqtn_satisfied_grad}.
	Since \fref{Corollary}{cor:phases_are_open} tells us that $ \left\{ M < \partial_3 p \right\} $ is an open set we may test \eqref{eq:eqtn_satisfied_grad}
	against any $\varphi\in C_c^\infty (M < \partial_3 p)$ to obtain
	\begin{equation*}
		\int_{ \left\{ M < \partial_3 p \right\} } \nabla_h u \cdot\nabla_h \varphi + \frac{1}{2} ( \partial_3 u)(\partial_3 \varphi)
		= - \langle \partial_i \PV,\, \varphi \rangle - \int_{ \left\{ M < \partial_3 p \right\} } \frac{1}{2} ( \partial_i M ) ( \partial_3 \varphi )
	\end{equation*}
	such that, by passing to the limit from $C_c^\infty$ to $H_0^1$, we obtain that
	$u\in H^1$ is an $H_0^1$-weak solution of
	\begin{equation*}
		\left( - \Delta_h - \frac{1}{2} \partial_3^2 \right) u = - \partial_i \PV + \frac{1}{2} \partial_3 \partial_i M \text{ in } \left\{ M < \partial_3 p \right\}
	\end{equation*}
	(note that $u$ does \emph{not} belong to $H_0^1$, it is only the test functions that belong to $H_0^1$).
	By classical theory (the elliptic operator has \emph{constant coefficients} when restricted to a single phase, so we may use for example Theorem 3.13 of \cite{han_lin}) it follows that,
	since $ \frac{1}{2} \partial_3 \partial_i M - \partial_i \PV \in L^q \left( M < \partial_3 p \right)$,
	consequently $\nabla u\in C^{0,\,\alpha}_\text{loc} \left( M < \partial_3 p \right)$ for $\alpha = 1 - \frac{d}{q}$.
	Similarly we obtain that $u\in H^1$ is an $H_0^1$-weak solution of
	\begin{equation*}
		-\Delta u = - \partial_i \PV \text{ in } \left\{ \partial_3 p < M \right\}
	\end{equation*}
	such that, since $ \partial_i \PV \in L^q \left( \partial_3 p < M \right)$, 
	it follows that $\nabla u\in C^{0,\,\alpha}_\text{loc} \left( \partial_3 p < M \right)$
	as desired.
\end{proof}

We are now ready to prove the second main result of this section, namely identifying sufficient conditions on $\PV$ and $M$ which ensure that solutions of $\PV$-and-$M$ inversion
are classical in each phase.

\begin{theorem}[Classical solutions in each phase]
\label{thm:classical_sols_each_phase}
	Let $M\in H^1$ be continuous with both $\nabla M$ and $\nabla \partial_3 M$ belonging to $L^q \left( M < \partial_3 p \right)$,
	let $\PV \in L^q$, and also let $\nabla \PV \in L^q \left( M \neq \partial_3 p \right)$ for some $q > d = 3$.
	If $p$ is an $\mathring{H}^1$-weak solution of \eqref{eq:PV_M_inv_strong} then it is a classical solution of
	\begin{align}
	\label{eq:classical_sol_unsaturated}
		\left( - \Delta_h - \frac{1}{2} \partial_3^2 \right) p &= - \PV + \frac{1}{2} \partial_3 M \text{ in } \left\{ M < \partial_3 p \right\}
	\shortintertext{and}
	\label{eq:classical_sol_saturated}
		-\Delta p &= - \PV \text{ in } \left\{ \partial_3 p < M \right\}.
	\end{align}
\end{theorem}
\begin{proof}
	\fref{Corollary}{cor:phases_are_open} tells us that $ \left\{ M < \partial_3 p \right\}$ is open so we may test \eqref{eq:PV_M_inv_strong}
	against $\varphi\in C_c^\infty( M < \partial_3 p)$, obtaining
	\begin{align}
		- \int_{ \left\{ M < \partial_3 p \right\} } (\PV) \varphi
		= \int_{ \left\{ M < \partial_3 p \right\} } \nabla_h p \cdot \nabla_h \varphi + \frac{1}{2} (\partial_3 p) ( \partial_3 \varphi) + \frac{1}{2} M ( \partial_3 \varphi).
	\label{eq:classical_sols_each_phase_int}
	\end{align}
	The usual approximation argument then shows that \eqref{eq:classical_sols_each_phase_int} actually holds for all $\varphi\in H_0^1 (M<\partial_3 p)$.
	Therefore, since \fref{Lemma}{lemma:Holder_cty_Hessian} tells us that $p\in C^2 (M < \partial_3 p)$ while Morrey's embedding tells us that, for $q > d$,
	$
		\PV,\, \partial_3 M \in C^{0,\,1-d/q}_\text{loc} ( M < \partial_3 p)
	$
	we conclude that indeed $p$ is a classical solution of \eqref{eq:classical_sol_unsaturated}.
	In the same way we deduce from \eqref{eq:PV_M_inv_strong} and $\PV$ having average zero that
	\begin{equation*}
		\int_{ \left\{ \partial_3 p < M \right\} } \nabla p \cdot \nabla\phi + (\PV)\phi = 0
		\text{ for all } \phi\in H_0^1 ( \partial_3 p < M),
	\end{equation*}
	from which it follows as above that $p$ is a classical solution of \eqref{eq:classical_sol_saturated}.
\end{proof}

We now record the penultimate main result of this section, which consists in a higher-regularity result as well as in establishing that, if both $\PV$ and $M$ are smooth,
then so is, in each phase, the solution of $\PV$-and-$M$ inversion.

\begin{theorem}[Higher regularity in each phase]
\label{thm:higher_reg}
	Let $q > d = 3$, let $k\geqslant 0$ be an integer, and let $\alpha\in (0,\,1)$. Suppose that
	\begin{itemize}
		\item	$M\in H^1$ with $\nabla M,\, \nabla\partial_3 M \in L^q ( M <\partial_3 p )$ and $\partial_3 M \in C^{k,\,\alpha} ( M <\partial_3 p)$ and that
		\item	$\PV\in L^q (\mathbb{T}^3)$ with $\nabla \PV \in L^q ( M \neq \partial_3 p ) $ and $ \PV \in C^{k,\,\alpha} (M \neq \partial_3 p)$.
	\end{itemize}
	If $p$ is an $\mathring{H}^1$-weak solution of \eqref{eq:PV_M_inv_strong} then $p\in C^{k,\,\alpha}_\text{loc} ( M \neq \partial_3 p)$.
	In particular if $\partial_3 M$ and $\PV$ are $C^\infty$ in $ \left\{ M < \partial_3 p \right\}$ and $ \left\{ M \neq \partial_3 p \right\}$, respectively,
	then $p$ is $C^\infty$ in $\left\{ M \neq \partial_3 p \right\}$.
\end{theorem}
\begin{proof}
	Under these assumptions \fref{Theorem}{thm:classical_sols_each_phase} tells us that $p$ is a classical solution
	of \eqref{eq:classical_sol_unsaturated}--\eqref{eq:classical_sol_saturated}.
	The result then follows from classical Schauder interior estimates (see for example Theorem 6.17 of \cite{gilbarg_trudinger}).
\end{proof}

As we head towards the conclusion of this section we provide a recipe for the construction of explicit one-dimensional ($x_3$--dependent) solutions.
This construction is of particular interest since it can then be used to find examples where $\PV$ and $M$ are both smooth and yet the solution of $\PV$-and-$M$ inversion is not.
More precisely, that solution will be $C^{1,\,1}$ but not $C^2$, which shows that \fref{Theorem}{thm:Holder_cty_grad} above is nearly sharp.

Two remarks are in order.
First, note that below we specify that the period of $\PV$ and $M$ is $2\pi$.
This is done solely to fix notation and the same construction carries through for any period.
Second, note that we are actually only able to specify the \emph{profile} of $M$.
Indeed, the solution constructed satisfies $\PV$-and-$M$ inversion where the data is $\PV$ and $M-c$,
shifted by some constant $c$ which depends on $\PV$ and $M$.

\begin{lemma}[Explicit one-dimensional solutions]
\label{lemma:explicit_1d_sols}
	Suppose that $\PV = \PV(x_3)$ and $M = M(x_3)$ are smooth $2\pi$--periodic functions.
	Let $\phi_m (x) \vcentcolon= x + \frac{1}{2} \min_0 (m-x)$ where, as usual, $\min_0 (x) \vcentcolon= \min(x,\,0)$.
	$\phi_m$ is invertible for any $m\in\mathbb{R}$ and the following holds.
	Let $A$ be an antiderivative of $\PV$ such that $A' = \PV$, let $\widetilde{\Theta} \vcentcolon= \phi_M^{-1} \circ A$, and let $c \vcentcolon= \fint_{-\pi}^{\pi} \widetilde{\Theta}$.
	Then $p = p(x_3) \vcentcolon= \int_{-\pi}^{x_3} ( \widetilde{\Theta} - c )$ solves \eqref{eq:PV_M_inv_strong}
	with data $\PV$ and $M-c$.
\end{lemma}
\begin{proof}
	That $\phi_m$ is invertible follows directly from the fact that it is strictly increasing (since $\phi_m' \geqslant 1/2 > 0$).
	Nonetheless it will be useful to have an explicit representation of its inverse. A direct computation shows that its inverse is given by
	$\phi_m^{-1} (x) = x - \min_0 (m-x)$.
	We may then compute, since $\phi_{m-c} (x-c) = \phi_m(x) - c$, that the identity 
	$
		p' + \frac{1}{2} {\min}_0\, \left( (M-c) - p' \right)
		= \phi_{M-c} (p')
		= (\phi_M \circ \widetilde{\Theta}) - c
	$
	holds.
	Therefore
	\begin{align*}
		\Delta p + \frac{1}{2} \partial_3 \left( {\min}_0\, ((M-c)-\partial_3 p) \right)
		= \left[ p' + \frac{1}{2} {\min}_0\, \left( (M-c) - p' \right) \right] '
	\\
		= (\phi_M \circ \widetilde{\Theta} )'
		= A'
		= \PV,
	\end{align*}
	as desired.
\end{proof}

As an immediate consequence we may conclude this section with its last main result, verifying that the global regularity (including across the interface) deduced in
\fref{Theorem}{thm:Holder_cty_grad} is nearly sharp.
Once again we work with $2\pi$--periodic functions to fix notation but the same result holds for arbitrary periods.

\begin{cor}[Sharpness of the regularity theory]
\label{cor:sharp_reg}
	Consider a smooth periodic function $M = M(x_3) : (-\pi,\,\pi) \to \mathbb{R}$ such that
	$M'(0) \neq 0$ and $\sign M(x) = \sign x$ for every $x\in (-\pi,\,\pi)$.
	In particular this means that $M(0) = 0$ (and that $M'(0) > 0$).
	We introduce $c \vcentcolon= \fint_{-\pi}^{\pi} \min_0 M$ and define $\Theta \vcentcolon= c - \min_0 M$.
	Then the function
	$
		p = p (x_3) \vcentcolon= \int_{-\pi}^{x_3} \Theta
	$
	is a solution of \eqref{eq:PV_M_inv_strong} with data $\PV \equiv 0$ and $M+c$.
	Moreover the gradient of $p$ is Lipschitz continuous everywhere but fails to be differentiable at zero.
\end{cor}
\begin{proof}
	That $p$ is a solution of $\PV$-and-$M$ inversion follows immediately from \fref{Lemma}{lemma:explicit_1d_sols} since $-\min_0 M = \phi_M^{-1} (0)$
	(for $\phi$ as in \fref{Lemma}{lemma:explicit_1d_sols}).
	Moreover, since $\sign M = \sign x$, we have that
	\begin{equation*}
		p'(x) = \Theta (x) = c - \left\{
			\begin{aligned}
				&M(x)	&&\text{if } x < 0 \text{ and } \\
				&0	&&\text{if } x \geqslant 0.
			\end{aligned}
		\right.
	\end{equation*}
	We may then conclude that indeed $p'$ is Lipschitz continuous since $M(0) = 0$, and yet fails to be differentiable at zero since $M'(0) \neq 0$.
\end{proof}

A simple example of \fref{Corollary}{cor:sharp_reg} in action is provided below.

\begin{example}
\label{example:sharp_reg}
	Consider $M:(-\pi,\,\pi)\to\mathbb{R}$ given by $M(x_3) = \sin (x_3) - 1/\pi$. Then, for $\PV \equiv 0$, the function $p = p(x_3)$ given by
	\begin{equation*}
		p(x_3) = -\frac{x}{\pi} + \left\{
			\begin{aligned}
				&\cos x	&&\text{if } x \leqslant 0 \text{ and } \\
				&1	&&\text{if } x > 0
			\end{aligned}
		\right.
	\end{equation*}
	is a solution of $\PV$-and-$M$ inversion \eqref{eq:PV_M_inv_strong}. Indeed: $p'(x_3) = - \min ( \sin x_3,\, 0) - 1/\pi$ and hence
	$
		\left[ p' + \frac{1}{2} \min ( M - p',\, 0) \right] ' = 0,
	$
	as desired. In particular we see that $p$ is $C^{1,\,1}$ but not $C^2$.
	This so-called ``baseball cap'' example is depicted in \fref{Figure}{fig:sharp_reg}.

	Note that this example is constructed by following the recipe provided above in \fref{Lemma}{lemma:explicit_1d_sols},
	using $\sin x_3$ as profile for $M$ such that the constant $c$ is given by $c = - \fint_{-\pi}^\pi \min_0 (\sin x_3) \,dx_3 = 1/\pi$.

	\begin{figure}
		\centering
		\includegraphics{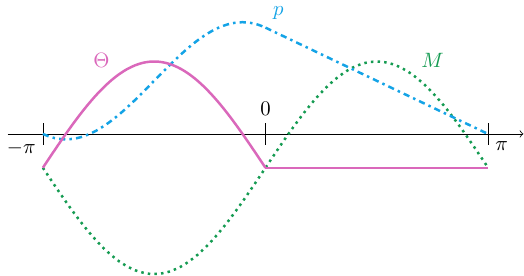}
		\caption{
			The ``baseball cap'': a very simple example where the gradient of $p$ (i.e. $\Theta$) is Lipschitz continuous but not differentiable: $M = \sin (x_3) - 1/\pi$ and $\PV = 0$.
			The name refers to the profile of the equivalent potential temperature $\Theta$.
			}
		\label{fig:sharp_reg}
	\end{figure}

\end{example}

	Note that with \fref{Corollary}{cor:sharp_reg} in hand we have now proved every part of \fref{Theorem}{thm:wp}, the main result of this paper.
\begin{proof}[Proof of \fref{Theorem}{thm:wp}]
	We combine \fref{Theorems}{thm:exist_and_unique}, \ref{thm:Holder_cty_grad}, and \ref{thm:higher_reg} and \fref{Corollaries}{cor:phases_are_open} and \ref{cor:sharp_reg}.
\end{proof}


\section{Smoothing}
\label{sec:smoothing}

The purpose of this section is two-fold.
First it details how using $\PV$-and-$M$ inversion in its form \eqref{eq:PV_M} which explicitly involves Heavisides makes it difficult to determine whether that inversion is an elliptic problem or not.
This is done by showing that so-called agnostic smoothing of the Heavisides,
which does not take into account their dependence on the moist variable $M$ and the vertical gradient of the pressure $\partial_3 p$,
fails to determine whether or not $\PV$-and-$M$ inversion is elliptic.

Second we introduce in this section a more careful way to smooth out the Heavisides which relies on the variational formulation of $\PV$-and-$M$ inversion.
The resulting regularised PDE has smooth solutions (provided $\PV$-and-$M$ are smooth) which converge to the solution of the unregularised $\PV$-and-$M$ inversion.
We will also observe that this process is reversible: given an appropriately regularised version of $\PV$-and-$M$ inversion we may reconstruct the corresponding energy.
We then conclude this section with a discussion of the fact that minimizers of the smoothed out energy converge to the minimizer of the true energy.

First we discuss the subtlety in determining the ellipticity of $\PV$-and-$M$ inversion.

\begin{remark}[Agnostic smoothing fails to determine whether or not $\PV$-and-$M$ inversion is elliptic]
\label{rmk:agnostic_smoothing}
	In order to streamline the computations we will begin with a purely algebraic manipulation of $\PV$-and-$M$ inversion in its form explicitly invoking Heavisides.
	Since $H_u + H_s = 1$, \eqref{eq:PV_M} may be rewritten as
	\begin{equation}
	\label{eq:PV_M_Heavi_only_Hu}
		\Delta p + \frac{1}{2} \partial_3 \left( (M - \partial_3 p) H_u \right) = \PV.
	\end{equation}
	We now suppose that $\mathcal{H}_u$ is a smooth approximation of $H_u$, remaining agnostic as to how precisely the approximation is made.
	We only posit that $\mathcal{H}_u = \mathcal{H}_u (M,\, \partial_3 p)$.
	This is sensible since the original Heavisides depend on the water content $q$, which itself may be written in terms of $M$ and $\partial_3 p$
	(the horizontal derivatives of $p$ play no part in the determination of $q$).
	We thus consider \begin{equation}
	\label{eq:PV_M_Heavi_only_Hu_reg}
		\Delta p + \frac{1}{2} \partial_3 \left( (M-\partial_3 p) \mathcal{H}_u \right) = \PV
	\end{equation}
	as a regularised version of \eqref{eq:PV_M_Heavi_only_Hu}.
	Since $\mathcal{H}_u = \mathcal{H}_u (M,\,\partial_3 p)$ is smooth we may compute that
	\begin{equation*}
		\partial_3 \left( (M - \partial_3 p) \mathcal{H}_u \right)
		= ( \partial_3 M - \partial_3^2 p ) \mathcal{H}_u
			+ (M - \partial_3 p ) \left(
				( \partial_M \mathcal{H}_u ) \partial_3 M
				+ ( \partial_p \mathcal{H}_u ) \partial_3^2 p
			\right)
	\end{equation*}
	such that \eqref{eq:PV_M_Heavi_only_Hu_reg} may be written equivalently as
	\begin{align*}
		\PV
		= \Delta_h p + \left( 1 - \frac{1}{2} \mathcal{H}_u + \frac{1}{2} (M - \partial_3 p) \partial_p \mathcal{H}_u \right) \partial_3^2 p
	\\
		+ \frac{1}{2} (\partial_3 M) \mathcal{H}_u + \frac{1}{2} (M - \partial_3 p) ( \partial_M \mathcal{H}_u ) \partial_3 M.
	\end{align*}
	The principal part is thus the operator
	\begin{equation*}
		\Delta_h + \left( 1 - \frac{1}{2} \mathcal{H}_u + \frac{1}{2} (M - \partial_3 p) \partial_p \mathcal{H}_u \right) \partial_3^2.
	\end{equation*}
	Here lies the difficulty: while sensible approximations would ensure that the regularised Heaviside satisfies $0 \leqslant \mathcal{H}_u \leqslant 1$
	such that $1 - \frac{1}{2} \mathcal{H}_u$ is uniformly bounded below by $1/2$,
	we know nothing about the sign of $ \frac{1}{2} (M - \partial_3 p) \partial_p \mathcal{H}_u$,
	and therefore know nothing about the sign of the coefficient of $\partial_3^2$.
	This means that we cannot say anything about the ellipticity of $\PV$-and-$M$ inversion with this approach.

    We emphasise that, as this remark details, this way of regularizing the PDE is certainly rather naive. The point is that if the variational structure underlying the PDE is not identified, this naive regularisation remains the best we can do, with all of its drawbacks detailed above.
\end{remark}

We now turn our attention towards a more careful way to smooth out the Heavisides, leveraging the variational formulation.
This smoothing relies on mollification -- see \fref{Definition}{def:mollifier} for the convention used here -- and at the crux of the computations
lies the expression for the smoothed out version of the function $\min ( \,\cdot\, ,\, 0)$ recorded below.
This mollification is also depicted in \fref{Figure}{fig:reg}.

\begin{lemma}[Explicit formulae for $ {\min}_\varepsilon $ and its derivative]
\label{lemma:explicit_form_min_epsilon}
	We introduce, for any $x\in\mathbb{R}$, the notation $\min_0 (x) \vcentcolon= \min (x,\,0)$ for $x\in\mathbb{R}$ and,
	for $\varphi_\varepsilon$ a \hyperref[def:mollifier]{standard mollifier},
	$\min_\varepsilon \vcentcolon= \min_0 \ast \varphi_\varepsilon$.
	If $\varphi_\varepsilon$ is \hyperref[def:mollifier]{centered} then
	\begin{equation}
	\label{eq:explicit_form_min_epsilon}
		{\min}_\varepsilon (x) = \left\{
		\begin{aligned}
			& x			&&\text{if } x \leqslant - \varepsilon,\\
			& F_\varepsilon (x)	&&\text{if } \lvert x \rvert < \varepsilon, \text{ and } \\
			& 0			&&\text{if } x \geqslant \varepsilon
		\end{aligned}
		\right.
	\end{equation}
	where, for $ \lvert x \rvert < \varepsilon$,
	\begin{equation}
	\label{eq:explicit_form_min_epsilon_F}
		F_\varepsilon (x) \vcentcolon= \frac{x}{2} + \frac{1}{2} \int_x^{\varepsilon} (x-y) \varphi_\varepsilon (y) dy - \frac{1}{2} \int_{-\varepsilon}^x (x-y) \varphi_\varepsilon (y) dy
	\end{equation}
\end{lemma}
\begin{proof}
	We begin by noting that, since $\supp\varphi_\varepsilon \subseteq (-\varepsilon,\, \varepsilon)$,
	\begin{equation}
	\label{eq:explicit_form_min_epsilon_int}
		{\min}_\varepsilon (x) = \int_{\max (-\varepsilon,\, x)}^\varepsilon (x-y) \varphi_\varepsilon (y) dy.
	\end{equation}
	To conclude the computation it thus suffices to split into three cases.
	\begin{itemize}
		\item	Case 1: $x\leqslant -\varepsilon$.
			Then $\max(-\varepsilon,\, x) = -\varepsilon$ and so, since $\varphi_\varepsilon$ is centered, \eqref{eq:explicit_form_min_epsilon_int}
			tells us that $\min_\varepsilon (x) = x$.
		\item	Case 2: $x \geqslant \varepsilon$.
			Then $\max(-\varepsilon,\, x) = x \geqslant \varepsilon$ and so the integral in \eqref{eq:explicit_form_min_epsilon_int} is taken over an empty set,
			such that $\min_\varepsilon (x) = 0$.
		\item	Case 3: $ \lvert x \rvert < \varepsilon$ -- a.k.a. the interesting case.
			Then $\max (-\varepsilon,\,x) = x$.
			In particular since $\int_x^\varepsilon = \int_{-\varepsilon}^\varepsilon - \int_{-\varepsilon}^x$ and since $ \varphi_\varepsilon  $ is centered we compute that
			\begin{equation*}
				{\min}_\varepsilon (x) = \int_x^\varepsilon (x-y) \varphi_\varepsilon (y) dy
				= x - \int_{-\varepsilon}^x (x-y) \varphi_\varepsilon (y) dy.
			\end{equation*}
			Symmetrizing these two expressions for $\min_\varepsilon$ produces, as desired
			\begin{equation*}
				{\min}_\varepsilon (x) = \frac{1}{2} \int_x^\varepsilon (x-y) \varphi_\varepsilon (y) dy + \frac{1}{2} \left(
					x - \int_{-\varepsilon}^x \varphi_\varepsilon (y) dy
				\right)
				= F_\varepsilon (x).
			\end{equation*}
	\end{itemize}
	The three cases above verify \eqref{eq:explicit_form_min_epsilon}.
\end{proof}

\begin{figure}
	\centering
	\includegraphics{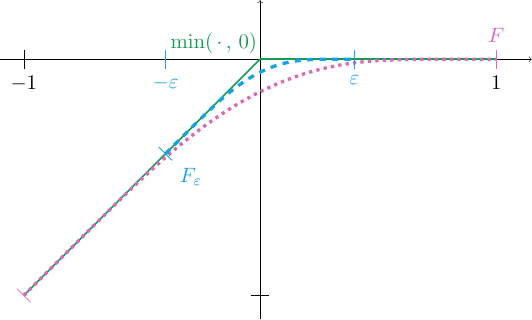}
	\caption{
		Following the notation of \fref{Lemma}{lemma:explicit_form_min_epsilon}, $\min_\varepsilon$ is the mollification of $\min_0 \vcentcolon= \min ( \,\cdot\, ,\, 0)$.
		The explicit form of the mollification is given in \fref{Lemma}{lemma:explicit_form_min_epsilon}
		in terms of $F_\varepsilon$ such that $\min_\varepsilon = F_\varepsilon$ in $(-\varepsilon,\, \varepsilon)$
		while $\min_\varepsilon = \min_0$ elsewhere. Here we show $\min_0$, the reference mollification $F = F_1$, and $F_\varepsilon$ for $0 < \varepsilon < 1$.
		Having an explicit expression for $\min_\varepsilon$ is the key element required to have an explicit form of regularised $\PV$-and-$M$ inversion
		(see \fref{Theorem}{thm:regularised_PV_and_M_inv}).
	}
	\label{fig:reg}
\end{figure}

With the expressions for $\min_\varepsilon$ and its derivative obtained in \fref{Lemma}{lemma:explicit_form_min_epsilon} above in hand
we may deduce the form taken by the smoothed out $\PV$-and-$M$ inversion.

\begin{theorem}[Regularised $\PV$-and-$M$ inversion]
\label{thm:regularised_PV_and_M_inv}
	Let $ \varphi_\varepsilon $ be a \hyperref[def:mollifier]{standard centered mollifier} such that, for
	$f_0 (s) \vcentcolon= \frac{1}{2} {\min (s,\,0)}^2$ we define $f_\varepsilon \vcentcolon= f_0 \ast \varphi_\varepsilon $.
	Then for $M\in L^2$, $\PV\in H^{-1}$, and $p\in\mathring{H}^1$ we define
	\begin{equation}
	\label{eq:reg_en}
		E_\varepsilon (p) \vcentcolon= \int_{\mathbb{T}^3} \frac{1}{2} {\lvert \nabla p \rvert}^2  - \frac{1}{2} f_\varepsilon (M - \partial_3 p) + \langle \PV,\, p \rangle.
	\end{equation}
	For $\min_\varepsilon$ defined in \fref{Lemma}{lemma:explicit_form_min_epsilon} we have that, for any $p,\,\phi\in \mathring{H}^1 $,
	\begin{equation}
	\label{eq:reg_Gateaux_deriv}
		DE_\varepsilon (p) \phi = \int_{\mathbb{T}^3} \nabla p \cdot \nabla\phi  - \frac{1}{2} {\min}_\varepsilon  (M - \partial_3 p) (\partial_3 \phi)
			+ \langle \PV,\, \phi \rangle
	\end{equation}
	such that $p\in \mathring{H}^1 $ is the minimizer of $E_\varepsilon$ if and only if $p$ is an $ \mathring{H}^1 $-weak solution of
	\begin{equation}
	\label{eq:regularised_PV_and_M_inv}
		-\Delta p - \frac{1}{2} \partial_3 \left( {\min}_\varepsilon (M-\partial_3 p) \right) = - \PV.
	\end{equation}
	In particular if we introduce
	\begin{equation}
	\label{eq:reg_Heavisides}
		H_u^\varepsilon \vcentcolon= \mathds{1} ( M - \partial_3 p \leqslant \varepsilon )
		\text{ and } 
		H_{int}^\varepsilon \vcentcolon= \mathds{1} ( - \varepsilon < M - \partial_3 p < \varepsilon )
	\end{equation}
	then we may write \eqref{eq:regularised_PV_and_M_inv} as
	\begin{equation}
	\label{eq:regularised_PV_and_M_inv_with_Heavisides}
		-\Delta p - \frac{1}{2} \partial_3 \left(
			(M - \partial_3 p) H_u^\varepsilon + F_\varepsilon (M-\partial_3 p) H_{int}^\varepsilon
		\right),
	\end{equation}
	for $F_\varepsilon$ defined in \eqref{eq:explicit_form_min_epsilon_F}.
\end{theorem}
\begin{proof}
	\eqref{eq:reg_Gateaux_deriv} is obtained from differentiating \eqref{eq:reg_en} since $f_\varepsilon ' = f_0' \ast \varphi_\varepsilon = \min_\varepsilon$.
	Integrating by parts then yields \eqref{eq:regularised_PV_and_M_inv} and using \fref{Lemma}{lemma:explicit_form_min_epsilon} produces yields \eqref{eq:regularised_PV_and_M_inv_with_Heavisides}.
\end{proof}

We now turn our attention to the matter of \emph{reversibility}: given an appropriately smoothed out version of $\PV$-and-$M$ inversion,
can we reconstruct the corresponding (smooth) energy? The answer is yes.
In \fref{Lemma}{lemma:prop_F} we begin by recording some properties of the smoothed out version of $\PV$-and-$M$ inversion,
then in \fref{Theorem}{thm:reversibility} we show that if these properties hold for \emph{any} smoothed out version of $\PV$-and-$M$ inversion, then a smooth energy can be reconstructed
via mollification.

\begin{lemma}[Properties of $F$]
\label{lemma:prop_F}
	Let $ \varphi_\varepsilon $ be a \hyperref[def:mollifier]{standard centered mollifier} and consider $F_\varepsilon$ as in \fref{Lemma}{lemma:explicit_form_min_epsilon}.
	Since $F_\varepsilon = \varepsilon F \left( \frac{\cdot}{\varepsilon} \right)$
	we define $F \vcentcolon= F_1$.
	Then $F$ satisfies
	\begin{equation}
	\label{eq:prop_F}
		F(-1) = -1,\,
		F'(-1) = 1,\,
		F(1) = F'(1) = 0,\,
		F''\leqslant 0, \text{ and } 
		\supp F'' \subseteq (-1,\,1).
	\end{equation}
\end{lemma}
\begin{proof}
	The identities $F(-1) = -1$ and $F(1) = 0$ follow immediately from $ \varphi$ being \hyperref[def:mollifier]{normalised} and centered and
	from the definition of $F_\varepsilon$ in \eqref{eq:explicit_form_min_epsilon_F}.
	Similarly the identities $F'(-1) = 1$ and $F'(1) = 0$ follow from $ \varphi$ being normalised.
	Differentiating \eqref{eq:explicit_form_min_epsilon_F} twice tells us that $F'' = -\varphi$ and so that fact that $F'' \leqslant 0$ follows from the non-negativity of $\varphi$.
\end{proof}

\begin{remark}[Demystifying the relation between $\varphi$ and $F$]
\label{rmk:demistify_rel_F_varphi}
	In the proof of \fref{Lemma}{lemma:prop_F} above we have obtained that $F'' = -\varphi$.
	This provides an alternative route to \eqref{eq:explicit_form_min_epsilon_F} expressing $F$ in terms of $\varphi$.
	Indeed, by taking the mean of the representations of $F(x)$ using a Taylor expansion about $-1$ and $1$, with the remainder in integral form, tells us that
	\begin{align*}
		F(x) = \frac{F(-1) + F(1)}{2} + \frac{F'(-1) - F'(1)}{2} + \frac{F'(-1) + F'(1)}{2} x
	\\
			+ \frac{1}{2} \int_{-1}^x (x-y) F''(y) dy - \frac{1}{2} \int_x^1 (x-y) F''(y) dy.
	\end{align*}
	Note that this identity holds for \emph{any} twice continuously differentiable function.
	In particular, if \eqref{eq:prop_F} holds and $F'' = -\varphi$ then we recover \eqref{eq:explicit_form_min_epsilon_F}.
\end{remark}

We now state and prove the reversibility theorem.

\begin{theorem}[Reversibility]
\label{thm:reversibility}
	Suppose that a smooth function $F$ satisfies \eqref{eq:prop_F}.
	Then there exists a \hyperref[def:mollifier]{standard centered mollifier} $\varphi$, given by $\varphi = -F''$,
	such that if we define $E_\varepsilon$ as in \eqref{eq:reg_en},
	$H_u^\varepsilon$ and $H_\text{int}^\varepsilon$ as in \eqref{eq:reg_Heavisides},
	and let $F_\varepsilon \vcentcolon= \varepsilon F \left( \frac{\cdot}{\varepsilon} \right)$
	then, for any $p\in \mathring{H}^1 $,
	the strong form of $DE_\varepsilon (p) \phi = 0$ for every $\phi\in \mathring{H}^1 $ is precisely \eqref{eq:regularised_PV_and_M_inv_with_Heavisides}.
\end{theorem}
\begin{proof}
	All that we need to do is verify that $\varphi$ is indeed a standard centered mollifier.
	Clearly $\varphi = - F''$ is smooth, supported on $(-1,\,1)$, and since $F''\leqslant 0$ by assumption we know that $\varphi$ is non-negative.
	The fact that $\varphi$ is \hyperref[def:mollifier]{normalised} follows from the values of $F'$ at $\pm 1$,
	and similarly the fact that $\varphi$ is centered follows from the identity $\int yF'' = yF' - \int F'$ and the values of $F$ and $F'$ at $\pm 1$.
\end{proof}

To conclude this section we turn our attention towards the proof that minimizers of the smoothed out energy converge to the minimizer of the true energy.
In order to do so we first record the following result concerning the derivative of the smoothed out energy.

\begin{lemma}[The derivative of the $\varepsilon$-regularisation is $\varepsilon$-close to the true derivative]
\label{lemma:grad_reg_epsilon_close_true_grad}
	Let $M\in L^2$, $\PV\in H^{-1}$, and let $E$ be the energy introduced in \fref{Definition}{def:energy}.
	Let $\varphi$ be a \hyperref[def:mollifier]{standard mollifier} and let $E_\varepsilon$ be as defined in \fref{Theorem}{thm:regularised_PV_and_M_inv}.
	For any $p\in \mathring{H}^1 $, if we write $C_\varphi \vcentcolon= \int \lvert y \rvert \varphi (y) dy$ then
	$
		\norm{ DE(p) - DE_\varepsilon (p)}{H^{-1}} \leqslant C_\varphi {\lvert \mathbb{T}^3 \rvert}^{1/2} \varepsilon / 2.
	$
\end{lemma}
\begin{proof}
	As in \fref{Theorem}{thm:regularised_PV_and_M_inv} we write $f_0 (s) \vcentcolon= \frac{1}{2} \min {(s,\,0)}^2$ and $f_\varepsilon \vcentcolon= f_0 \ast \varepsilon$,
	noting that $f_0 ' = \min ( \,\cdot\, ,\, 0 )$.
	Subtracting the expression for $DE_\varepsilon$ recorded in \eqref{eq:regularised_PV_and_M_inv} from the expression for $DE$ recorded in \eqref{eq:Gateaux_diff_en} then tells us that,
	for any $\phi\in \mathring{H}^1 $,
	$
		( DE(p) - DE_\varepsilon (p)) \phi
		= \frac{1}{2} \int_{\mathbb{T}^3} ( f_0' - f_\varepsilon ')(M - \partial_3 p) (\partial_3 \phi).
	$
	Since $\text{Lip}(f_0') = 1$ we then deduce from H\"{o}lder's inequality and \fref{Lemma}{lemma:unif_conv_moll_Lip_func} that
	$
		\lvert ( DE(p) - DE_\varepsilon (p)) \phi \rvert
	 	\leqslant C_\varphi \varepsilon { \lvert \mathbb{T}^3 \rvert }^{1/2} \norm{\phi}{ \mathring{H}^1 }/2,
	$
	from which the claim follows.
\end{proof}

Finally we conclude this section by establishing that minimizers of the smoothed out energy converge to the minimizer of the true energy.

\begin{theorem}[Minimizers of the $\varepsilon$-energy are $\varepsilon$-close to the minimizer of the true energy]
\label{thm:min_of_reg_en_conv_to_true_min}
	Let $M\in L^2$, $\PV\in H^{-1}$, and let $E$ be the energy introduced in \fref{Definition}{def:energy}.
	Let $\varphi$ be a \hyperref[def:mollifier]{standard mollifier} and let $E_\varepsilon$ be as defined in \fref{Theorem}{thm:regularised_PV_and_M_inv}.
	Denote
	$
		p^* \vcentcolon= \argmin_{ \mathring{H}^1 } E
	$
	and
	$
		p^*_{\varepsilon} = \argmin_{ \mathring{H}^1 } E_\varepsilon
	$
	and let
	$
		C_\varphi \vcentcolon= \int \lvert y \rvert \varphi (y) dy.
	$
	Then
	$
		\norm{p^* - p^*_{\varepsilon}}{ \mathring{H}^1 } \leqslant \sqrt{2} C_\varphi { \lvert \mathbb{T}^3 \rvert }^{1/2} \varepsilon.
	$
\end{theorem}
\begin{proof}
	Since $DE(p^*) = 0$ the $ \frac{1}{2} $--strong convexity of $E$, proved in \fref{Proposition}{prop:prop_en}, in the form of \eqref{eq:equiv_charac_strg_conv_2}
	combined with \fref{Corollary}{cor:strong_convex_is_norm_near_minimizer} tells us that
	\begin{equation*}
		\norm{ p^* - p_\varepsilon^* }{ \mathring{H}^1 }^2
		\leqslant 4 \left[ E ( p^*\varepsilon ) - E( p^* ) \right]
		\leqslant \norm{ DE ( P_\varepsilon^* ) }{H^{-1}}^2.
	\end{equation*}
	We may then use \fref{Lemma}{lemma:grad_reg_epsilon_close_true_grad} to conclude since $\nabla_{ \mathring{H}^1 } E_\varepsilon \left( p^*_{\varepsilon} \right) = 0$ and hence,
	using the inequality above,
	\begin{align*}
		\norm{p^* - p^*_{\varepsilon}}{ \mathring{H}^1 }^2
		\leqslant 8 \left(
			\norm{ D E_\varepsilon \left( p^*_{\varepsilon} \right) }{ H^{-1} }^2
			+ \norm{ DE_\varepsilon \left( p^*_{\varepsilon} \right) - D E \left( p^*_{\varepsilon} \right) }{ H^{-1} }^2
		\right)
		\leqslant 2 C_\varphi^2 \lvert \mathbb{T}^3 \rvert \varepsilon^2,
	\end{align*}
	as claimed.
\end{proof}

\appendix

\section{Tools from convex analysis}
\label{sec:tools_convex_ana}

In this section we record various results from convex analysis that are of use to us.
This begins with results on various characterisations of strong convexity
and concludes with a lemma relating convexity to weak lower semi-continuity of functionals.

First we recall a few first-order characterisations of strong convexity.

\begin{lemma}[Equivalent characterisations of strong convexity]
\label{lemma:equiv_charac_strg_conv}
	Let $ \left( H,\, \langle \,\cdot\,,\,\cdot\, \rangle  \right)$ be a Hilbert space and let $f: H\to \mathbb{R}$ be Fr\'{e}chet differentiable.
	For any $\mu > 0$ the following are equivalent, where each inequality holds for every $x,\,y\in H$ and $\theta\in [0,\,1]$.
	\begin{align}
		&f(\theta x + (1-\theta) y) \leqslant \theta f(x) + (1-\theta) f(y) - \theta(1-\theta) \frac{\mu}{2} \norm{x-y}{}^2	\label{eq:equiv_charac_strg_conv_1}\\
		&f(y) \geqslant f(x) + \langle Df(x),\, y-x \rangle + \frac{\mu}{2} \norm{x-y}{}^2					\label{eq:equiv_charac_strg_conv_2}\\
		&\langle Df(x) - Df(y) ,\, x-y \rangle \geqslant \mu \norm{x-y}{}^2							\label{eq:equiv_charac_strg_conv_3}
	\end{align}
	If any of these inequalities hold recall that we say that $f$ is \emph{$\mu$-strongly convex}.
\end{lemma}
\begin{proof}
	See Definition 2.1.2 and Theorem 2.1.9 in \cite{nesterov}.
\end{proof}

Similarly we now recall some second-order characterisations of strong convexity.

\begin{lemma}[Second-order characterisation of strong convexity]
\label{lemma:second_order_charac_strg_conv}
	Let $ \left( H,\, \langle \,\cdot\,,\,\cdot\, \rangle  \right)$ be a Hilbert space and let $f: H\to \mathbb{R}$ be twice continuously Fr\'{e}chet differentiable.
	$f$ is $\mu$-strongly convex if and only if $D^2 f \geqslant 0$, in the sense that $D^2 f (x) (\xi,\,\xi) \geqslant 0$ for every $x,\,\xi\in H$.
\end{lemma}
\begin{proof}
	See Theorem 2.1.11 in \cite{nesterov}.
\end{proof}

We now record a simple but very useful consequence of strong convexity which tells us that the derivatives of strongly convex functions essentially act as norms near their minimizers.

\begin{cor}[Derivative of a strongly convex function is essentially a norm near its minimizer]
\label{cor:strong_convex_is_norm_near_minimizer}
	Let $ \left( H,\, \langle \,\cdot\,,\,\cdot\, \rangle \right)$ be a Hilbert space with dual $H^*$ and
	let $f:H\to\mathbb{R}$ be a Fr\'{e}chet differentiable and $\mu$-strongly convex.
	Then
	$
		f(x) - f(y) \leqslant \frac{1}{2\mu} \norm{ Df(x) }{H^*}^2
	$
	for every $x,\,y\in H$.
\end{cor}
\begin{proof}
	We write $\nabla_H f(x)$ for the Riesz representative of $D f(x)$ and note that
	$y^* = x - \frac{1}{2\mu} \nabla_H f(x)$ is the minimum of the right-hand side of the characterisation \eqref{eq:equiv_charac_strg_conv_2} of strong convexity with respect to $y$.
	Substituting $y = y^*$ into the right-hand side of \eqref{eq:equiv_charac_strg_conv_2} then tells us that, for every $x,\,y\in H$,
	\begin{equation*}
		f(y)
		\geqslant f(x) - \frac{1}{\mu} Df(x) \nabla_H f(x) + \frac{1}{2\mu} \norm{\nabla_H f(x)}{}^2
		= f(x) - \frac{1}{2\mu} \norm{Df(x)}{H^*}^2
	\end{equation*}
	as desired since $\norm{\nabla_H f(x)}{} = \norm{Df(x)}{H^*}$.
\end{proof}

Below we now record a result used for the existence and uniqueness result of \fref{Theorem}{thm:exist_and_unique}, i.e. one of the ingredients of the Direct Methods of the Calculus of Variations.
	
\begin{lemma}[Convex energy density implies weak lower semi-continuity]
\label{lemma:cvx_en_implies_wk_lsc}
	Consider $f:\mathbb{T}^3 \times \mathbb{R}^3 \to \mathbb{R}$ which satisfies the following conditions.
	\begin{enumerate}
		\item	For every $x\in\mathbb{T}^3$, $f(x, \,\cdot\,)$ is convex and continuously differentiable.
		\item	For every $u\in L^2 (\mathbb{T}^3,\, \mathbb{R}^3)$,
			\begin{equation*}
				x \mapsto f(x,\, u(x)) \in L^1 (\mathbb{T}^3) \text{ and }
				x \mapsto (\nabla_u f)(x,\, u(x)) \in L^2 (\mathbb{T}^3).
			\end{equation*}
	\end{enumerate}
	The functional $\mathcal{F}: \mathring{H}^1 \to \mathbb{R}$ defined by
	$
		\mathcal{F} (p) \vcentcolon= \int_{\mathbb{T}^3} f(x,\, \nabla p (x) ) dx
	$
	for every $p\in\mathring{H}^1$ is weakly lower semi-continuous in $\mathring{H}^1$.
\end{lemma}
\begin{proof}
	Let $(p_n)$ be a sequence in $\mathring{H}^1$ weakly convergent to $p\in\mathring{H}^1$.
	Since $f(x, \,\cdot\,)$ is convex and continuously differentiable we have that
	\begin{equation*}
		f(x,\, \nabla p) \leqslant f(x,\, \nabla p_n) + (\nabla_u f)(x,\,\nabla p) \cdot (\nabla p - \nabla p_n).
	\end{equation*}
	In particular, since 
	$
		(\nabla_u f)(x,\,\nabla p) \in L^2
	$
	and
	$
		\nabla p - \nabla p_n \rightharpoonup 0 \text{ in } L^2,
	$
	we deduce that indeed
	\begin{equation*}
		\mathcal{F} (p) = \int_{\mathbb{T}^3} f(x,\,\nabla p)
		\leqslant \liminf \int_{\mathbb{T}^3} f(x,\, \nabla p_n) + 0
		= \liminf \mathcal{F}(p_n),
	\end{equation*}
	as desired.
\end{proof}


\section{Tools from regularity theory}
\label{sec:tools_reg_th}

In this section we record tools from regularity theory that are of use to us.
First we discuss elementary properties of finite differences (used to obtain the $H^2$ regularity of solutions in \fref{Lemma}{lemma:H_2_reg}).
The bulk of this section is then devoted to the discussion of the de Giorgi regularity theory, and specifically its formulation in terms of Campanato functions.
Finally this section concludes with a simple result pertaining to approximations using mollifiers.

First we fix notation and recall the definition of finite differences.

\begin{definition}[Finite differences]
\label{def:fin_diff}
	We define, for $f \in L^1_\text{loc} (\mathbb{T}^3)$, $x,\,h\in\mathbb{T}^3$, and $i = 1,\,\dots,\,n$ the \emph{finite difference}
	$
		\Delta_h^i f(x) \vcentcolon= f(x + he_i) - f(x).
	$
\end{definition}

We now record elementary properties of finite differences.

\begin{lemma}[Properties of finite differences]
\label{lemma:fin_diff}
	For \hyperref[def:fin_diff]{finite differences} the following hold.
	\begin{enumerate}
		\item	Adjoint: $ \int_{\mathbb{T}^3} (\Delta_h^i f) g = \int_{\mathbb{T}^3} f ( \Delta_{-h}^i g)$.
		\item	Commutation: $ \partial_j $ and $\Delta_h^i $ commute.
		\item	Chain Rule: if $f$ is absolutely continuous then the following identity holds:
			$
				\Delta_h^i (f\circ g) = (\Delta_h^i g) \int_0^1 f' \left( ( \id + \theta\Delta_h^i ) g \right) d\theta.
			$
		\item	Characterisation of $W^{1,\,p}$, necessity: If $f\in W^{1,\,p} (\mathbb{T}^3)$ for $p\in [1,\,\infty)$ then we have that
			$
				\norm{\Delta_h^i f}{L^p} \leqslant C_d \lvert h \rvert \, \norm{\nabla f}{L^p}.
			$
		\item	Characterisation of $W^{1,\,p}$, sufficiency: If $u\in L^p (\mathbb{T}^3)$, $p\in (1,\,\infty)$, such that
			$
				\norm{\Delta_h^i u}{L^p} \leqslant M \lvert h \rvert
			$
			for some constant $M > 0$ then
			$u\in W^{1,\,p} (\mathbb{T}^3)$, we have the estimate
			$ \norm{\nabla u}{L^p} \leqslant M$, and
			$\frac{1}{h} \Delta_h^i u \rightharpoonup \partial_i u$ in $L^p$ as $h\to 0$.
		\item	$H^{-1}$ estimate: If $u\in L^2$ then $\Delta_h^i f \in H^{-1}$ and
			$
				\norm{\Delta_h^i f}{H^{-1}} \leqslant \lvert h \rvert \, \norm{f}{L^2}.
			$
	\end{enumerate}
\end{lemma}
\begin{proof}
	Item 1 follows from a direct computation.
	Item 2 follows from the fact that $\Delta_h^i f = f\circ\Phi_h^i - f$ for $\Phi_h^i(x) \vcentcolon= x + he_i$, such that $D\Phi_h^i = I$.
	Items 4 and 5 are proved in \cite{giaquinta_martinazzi}.
	Item 6 follows from the fact that, for any $f\in L^2$, $\nabla f\in H^{-1}$ with $ \norm{\nabla f}{H^{-1}} \leqslant \norm{f}{L^2}$.
	Now we finally we turn our attention to the proof of item 3, which comes down to a short computation.
	Writing $y \vcentcolon= g(x)$ and $z \vcentcolon= g(x+he_i)$, such that $z-y = (\Delta_h^i g)(x)$, we have that
	\begin{align*}
		\left[ \Delta_h^i (f\circ g) \right] (x)
		= f(z) - f(y)
		= \int_0^1 (z-y) f' \left( y + \theta(z-y) \right) d\theta
	\\
		= (\Delta_h^i g)(x) \int_0^1 f' \left( g(x) + \theta (\Delta_h^i g)(x) \right) d\theta,
	\end{align*}
	as claimed.
\end{proof}

We now turn our attention towards the topic occupying the bulk of this section: formulating the de Giorgi regularity theory using Campanato functions.
There are three important components of that in this section:
(1) obtaining a sufficient condition for H\"{o}lder continuity in terms of local integrals of gradients, which is done in \fref{Proposition}{prop:suff_Campanato_cond_gradients},
(2) reformulating the classical de Giorgi result in terms of Campanato functions, which is done in \fref{Proposition}{prop:Campanato_formulation_hom_DG},
and (3) deriving an \emph{a priori} estimate of Campanato-type for elliptic equations with forcing in divergence-form.

First we fix notation and define Campanato functions.

\begin{definition}[Campanato functions]
\label{def:Campanato}
	We say that $u\in L^2 (\mathbb{T}^d)$ is \emph{Campanato (regular) of order $\alpha$} for some $\alpha\in (0,\,1)$ if there exists a constant $K > 0$ such that,
	for every $x\in\mathbb{R}^3$ and $R>0$,
	$
		\int_{B(x,\,R)} {\lvert u - u_{x,\,R} \rvert}^2 \leqslant K^2 R^{d+2\alpha}
	$
	where
	$
		u_{x,\,R} \vcentcolon= \fint_{B(x,\,R)} u
	$
	denotes the local average of $u$.
	We call $K$ the \emph{Campanato constant} of $u$.
\end{definition}

\begin{remark}[Interpretation of the Campanato condition]
\label{rmk:interpretation_Campanato}
	It is interesting to consider the form of the defining estimate of Campanato functions using \emph{normalised averages}
	and we do this here in the (slightly more general) context of $L^p$ based Campanato spaces, for $p \in [1,\,\infty)$.
	Specifically, note that the defining inequality
	$
		\int_{B(x,\,R)} {\lvert u - u_{x,\,R} \rvert}^2 \leqslant M^p R^{d + \alpha p}
	$
	is equivalent to
	$
		{ \left( \fint_{B(x,\,R)} {\lvert u - \overline{u}_{x,\,R} \rvert}^p  \right)}^{1/p}
		\leqslant \frac{M}{\omega_d^{1/p}} R^\alpha,
	$
	where $\omega_d$ denotes the $d$-dimensional volume of the unit ball in $\mathbb{R}^d$.
	The latter inequality makes it more apparent that the deviation of $u$ from its average is only controlled by $R^\alpha$,
	as would be expected from a function which is $\alpha$-H\"{o}lder continuous.
\end{remark}

We now record half of the key fact about Campanato function, which is that they are nothing more than an equivalent characterisation of H\"{o}lder continuous functions.
We only mention one direction of the equivalence below, the other direction is much easier to prove.

\begin{lemma}[Campanato implies H\"{o}lder]
\label{lemma:Campanato_implies_Holder}
	Let $u \in L^2 (\mathbb{T}^d)$ be \hyperref[def:Campanato]{Campanato of order $\alpha$}, for some $\alpha\in (0,\,1)$, with \hyperref[def:Campanato]{Campanato constant} $K$. 
	Then $u$ is H\"{o}lder continuous with exponent $\alpha$, and for every open set $\mathcal{U}\subseteq\mathbb{T}^3$,
	$
		\norm{u}{ C^\alpha ( \mathcal{U} ) }
		\leqslant
		C \left( K + \norm{u}{ L^2 } \right)
	$
	for some constant $C = C(\alpha,\,d,\,\mathcal{U})$.
\end{lemma}
\begin{proof}
	See Theorem 3.1 of \cite{han_lin}
\end{proof}

We now progress towards the aforementioned sufficient condition for H\"{o}lder continuity involving local integrals of gradients.
The missing ingredient is a careful analysis of the scaling of the Poincar\'{e} constant in balls.

\begin{lemma}[Scaling of the Poincar\'{e} constant in balls]
\label{lemma:scaling_Poincare_constant}
	There exists a constant $C = C(d) > 0$ such that, for every $x\in\mathbb{T}^d$, every $r>0$, and every $u\in H^1( (B(x,\,r))$, the inequality 
	$
		\int_{B(x,\,r)} {\big\vert u - \fint_{B(x,\,r)} u \big\rvert}^2 \leqslant C r^2 \int_{B(x,\,r)} {\lvert \nabla u \rvert}^2
	$
	holds.
\end{lemma}
\begin{proof}
	This follows from the Poincar\'{e} inequality in the unit ball and scaling: defining $v (y) \vcentcolon= u(x + ry)$ such that $v\in H^1 (B(0,1))$
	and applying the Poincar\'{e} inequality to $v$ produces the claim.
\end{proof}

We are now ready to state and prove the first of three main results of this section.

\begin{prop}[Sufficient Campanato condition involving local integrals of gradients]
\label{prop:suff_Campanato_cond_gradients}
	Suppose that $u\in H^1_\text{loc} ( \mathbb{T}^d )$ such that, for some $K > 0$ and $\alpha\in (0,\,1)$, the inequality
	$
		\int_{B(x,\,r)} {\lvert \nabla u \rvert}^2
		\leqslant K^2 r^{d+2\alpha - 2}
	$
	holds for every $B(x,\,r) \subseteq \mathbb{T}^d$.
	Then $u\in C^\alpha$ and, for every open set $\mathcal{U} \subseteq\mathbb{T}^d$, we have the estimate
	$
		\norm{u}{C^\alpha ( \mathcal{U} )} \leqslant C \left( K + \norm{u}{L^2 (\mathbb{T}^d)} \right)
	$
	for some constant $C = C \left(\alpha,\,d,\,\mathcal{U} \right)$.
\end{prop}
\begin{proof}
	This follows from combining \fref{Lemma}{lemma:Campanato_implies_Holder}, which says that all Campanato functions are H\"{o}lder, with the Poincar\'{e} inequality
	in a ball -- c.f. \fref{Lemma}{lemma:scaling_Poincare_constant}.
\end{proof}

We now turn our attention towards proving the second of three main results of this section, namely reformulating the classical de Giorgi Theorem in terms of Campanato functions.
First we need this technical lemma about test functions, which really tells us that we can find a smooth version $\varphi$ of the indicator $\mathds{1}_{B_r}$
which vanishes outside the larger ball $B_R$ while retaining quantitative control over the gradient of $\varphi$.

\begin{lemma}[Careful scaling of a test function]
\label{lemma:careful_scale_test_func}
	There exists a universal constant $C>0$ such that for every $0 < r < R$ we may find $\varphi\in C_c^\infty (B_R)$ satisfying $0 \leqslant \varphi \leqslant 1$
	such that $\varphi \equiv 1$ in $B_r$ and
	$
		\norm{\nabla\varphi}{\infty} \leqslant \frac{C}{R-r}
	$.
\end{lemma}
\begin{proof}
	Consider a radial profile $\rho\in C^\infty (\mathbb{R})$ satisfying $0\leqslant \rho \leqslant 1$ which is supported in $(-\infty,\,1)$ and such that $\rho(s) = 1$ if $s<0$.
	The function $\varphi (x) \vcentcolon= \rho \left( \frac{ \lvert x \rvert - r}{R-r} \right)$ satisfies the desired properties, with $C = \norm{\rho'}{\infty}$.
\end{proof}

Since we will refer to uniformly elliptic matrices repeatedly in the sequel we fix notation on the matter here.

\begin{definition}[Uniform ellipticity]
\label{def:unif_ell}
	Let $\mathcal{U} \subseteq \mathbb{T}^3$.
	A matrix-valued pointwise symmetric function $A\in L^\infty$ is said to be \emph{uniformly elliptic with constants $\theta,\,\Lambda > 0$ in $\mathcal{U}$} if, for every $x\in \mathcal{U} $,
	$ A(x) \geqslant \theta I$ and $ \lvert A(x) \rvert \leqslant \Lambda$ for every $x\in\mathbb{T}^3$.
\end{definition}

We now continue building towards the Campanato formulation of the classical de Giorgi result -- the next step is to obtain the Caccioppoli inequality.
It will follow from the following energy estimate.

\begin{lemma}[Energy inequality, a.k.a. prelude to the Caccioppoli inequality]
\label{lemma:prelude_Caccioppoli}
	Let $u\in H^1$ be an $H^1$-weak solution of
	$
		-\nabla\cdot \left( A(x) \nabla u \right) = 0
	$ in $\mathbb{T}^d$
	where $A\in L^\infty$ is pointwise symmetric and \hyperref[def:unif_ell]{uniformly elliptic} with constants $\theta,\,\Lambda > 0$ in $\mathbb{T}^d$.
	For every $\varphi\in C_c^\infty$ the following estimate holds:
	$
		\int \varphi^2 {\lvert \nabla u \rvert}^2 \leqslant \frac{4\Lambda^2}{\theta^2} \int u^2 {\lvert \nabla\varphi \rvert}^2.
	$
\end{lemma}
\begin{proof}
	We may test $-\nabla\cdot (A\nabla u) = 0$ against $\varphi^2 u$ to observe that the identity
	$
		\int_{\mathbb{T}^3} \varphi^2 A\nabla u\cdot\nabla u = - 2 \int_{\mathbb{T}^3} A (\varphi\nabla u)\cdot (u\nabla\varphi)
	$
	holds.
	Using the ellipticity of $A$ on the left-hand side and an $\varepsilon$-Cauchy inequality on the right-hand side of this identity then allows us to conclude.
\end{proof}

We are now equipped to prove the Caccioppoli inequality.

\begin{cor}[Caccioppoli inequality]
\label{cor:Caccioppoli}
	Let $u\in H^1$ be an $H^1$-weak solution of
	$
		-\nabla\cdot \left( A(x)\nabla u \right) = 0
	$ in $\mathbb{T}^d$
	where $A\in L^\infty$ is pointwise symmetric and \hyperref[def:unif_ell]{uniformly elliptic} with constants $\theta,\,\Lambda > 0$ in $\mathbb{T}^d$.
	There exists a constant $C = C(d) > 0$ such that, for every $0 < r < R < \infty$ and every $a\in\mathbb{R}$,
	$
		\int_{B_r} {\lvert \nabla u \rvert}^2 \leqslant \frac{C \Lambda^2 }{ { \theta^2 \left( R - r \right) }^2 } \int_{B_R} {\lvert u - a \rvert}^2.
	$
\end{cor}
\begin{proof}
	Without loss of generality we may set $m=0$ since otherwise $v \vcentcolon= u - a$ satisfies $\nabla v = \nabla u$ and hence $v$ solves the same equation as $u$.
	We then pick $\varphi$ as in \fref{Lemma}{lemma:careful_scale_test_func} and use it as a test function in \fref{Lemma}{lemma:prelude_Caccioppoli} to establish the claim.
\end{proof}

We are now ready to reformulate the classical de Giorgi result in the language of Campanato.
First we record that classical theorem here.

\begin{theorem}[de Giorgi]
\label{thm:De_Giorgi_hom}
	Let $u\in H^1 (\Omega)$ be an $H^1$-weak solution of the equation
	$
		- \nabla\cdot \left( A(x)\nabla u \right) = 0
	$
	in a domain $\Omega\subseteq \mathbb{R}^d$ with $d\geqslant 3$,
	where $A\in L^\infty$ is pointwise symmetric and \hyperref[def:unif_ell]{uniformly elliptic} with constants $\theta,\,\Lambda > 0$ in $\Omega$.
	There exists $\alpha = \alpha (d,\, \theta,\, \Lambda) \in (0,\,1)$ and $C = C(d,\, \theta,\,\Lambda) > 0$ such that $u\in C^\alpha$ and, for every $x_0$,
	\begin{equation*}
		\sup_{x\in B \left( x_0,\, 1/4 \right)} \frac{ \lvert u(x) - u(x_0) \rvert}{ {\lvert x - x_0 \rvert}^\alpha} \leqslant C \norm{u}{L^2 ( B(x_0,\,2) )} .
	\end{equation*}
\end{theorem}
\begin{proof}
	This was first proved in \cite{de_giorgi}.
	Very digestible proofs are provided in \cite{cafarelli_vasseur, vasseur}.
\end{proof}

When reformulated using Campanato functions, de Giorgi's Theorem becomes the following result.

\begin{prop}[Campanato formulation of de Giorgi's Theorem]
\label{prop:Campanato_formulation_hom_DG}
	Let the function $w\in H^1 (B(x_0,\,R)) \subseteq \mathbb{T}^d$ be an $H^1$-weak solution of
	$
		-\nabla\cdot \left( A(x)\nabla w \right) = 0
	$
	in $B(x_0,\, R)$
	where $A\in L^\infty$ is pointwise symmetric and \hyperref[def:unif_ell]{uniformly elliptic} with constants $\theta,\,\Lambda$ in $B(x_0,\, R)$.
	There exist $\alpha = \alpha ( d,\,\theta,\,\Lambda) \in (0,\,1)$ and  $C = C(d,\,\theta,\,\Lambda) > 0$ such that, for any $0 < \rho < R$,
	$
		\int_{B(x_0,\,\rho)} {\lvert \nabla w \rvert}^2
		\leqslant C { \left( \frac{\rho}{R} \right)}^{d-2+2\alpha} \int_{B(x_0,\,R)} {\lvert \nabla w \rvert}^2.
	$
\end{prop}
\begin{proof}
	Note that if the claim is proved for $x_0 = 0$ and $R=2$ then we may deduce that it holds for arbitrary $x_0$ and $R$ by translation and scaling.
	Indeed, we may define $u(x) \vcentcolon= w \left( x_0 + \frac{R}{2} x \right)$ such that $u$ solves $-\nabla\cdot ( \widetilde{A}\nabla u) = 0$ in $B(0,\, 2)$,
	where crucially $\widetilde{A}(x) \vcentcolon= A \left( x_0 + \frac{R}{2} x \right)$ has the same uniform ellipticity constants as $A$.
	This allows us to deduce that if the claim holds for $w$ in $B(0,\,2)$, then it holds for $u$ in $B(x_0,\,R)$.

	So now we show that the claim holds when $x_0 = 0$ and $R=2$.
	Since $u - \fint u$ solves the same equation as $u$ we may without loss of generality assume that $\fint_{B_2} u = 0$.
	We now combine \fref{Corollary}{cor:Caccioppoli}, \fref{Theorem}{thm:De_Giorgi_hom}, and the Poincar\'{e} inequality:
	\begin{align*}
		\int_{B_\rho} {\lvert \nabla u \rvert}^2 
		\leqslant \frac{C}{\rho^2} \int_{B_{2\rho}} {\lvert u - u(0) \rvert}^2
		\leqslant \frac{C}{\rho^2} \int_{B_{2\rho}} {\lvert x \rvert}^{2\alpha} \left( \int_{B_2} u^2 \right)
		\leqslant C \rho^{d+2\alpha -2} \int_{B_2} {\lvert \nabla u \rvert}^2,
	\end{align*}
	where we have used that $\int_{B_r} {\lvert x \rvert}^p = C(d,\,p) s^{d+p}$.
	Note that the chain of inequalities above only holds if $\rho \leqslant 1/8$, such that $2\rho \leqslant 1/4$ and \fref{Theorem}{thm:De_Giorgi_hom} applies.
	Nonetheless, the desired inequality also holds for $\rho > 1/8$,
	albeit more trivially so since in that case we only need to note that $\rho^{d+2\alpha-2} \geqslant 8^{2-2\alpha-d} = C(\alpha,\,d)$,
	and this concludes the proof.
\end{proof}

We now record the third of three main results from this section: an \emph{a priori} estimate for elliptic equations whose forcing is in divergence-form.

\begin{prop}[Campanato-type \emph{a priori} estimate for forcing in divergence form]
\label{prop:Campanato_type_a_priori_estimate_div}
	Suppose that $d > 1$ and let $w\in H_0^1 (B(x_0,\,R))$ be an $H_0^1$-weak solution of
	$
		- \nabla\cdot \left( A\nabla w \right) = \nabla\cdot F
	$
	in $B(x_0,\,R) \subseteq \mathbb{T}^d$,
	where $A\in L^\infty$ is pointwise symmetric and \hyperref[def:unif_ell]{uniformly elliptic} with constants $\theta,\,\Lambda > 0$ in $\mathbb{T}^d$,
	and where $F\in L^q( B(x_0,\,R))$ for some $q>d$.
	There exist $\alpha = \alpha(d,\,q) \in (0,\,1)$ and $C = C(\alpha,\,d,\,\theta) > 0$ such that
	$
		\int_{B(x_0,\,R)} {\lvert \nabla w \rvert}^2 
		\leqslant C \norm{F}{L^q ( B(x_0,\,R) )} R^{d-2+2\alpha}.
	$
\end{prop}
\begin{proof}
	We choose $\tilde{q} \in (1,\,\infty)$ such that
	$
		\frac{1}{q} + \frac{1}{\tilde{q}} = \frac{1}{2}
	$
	, which is possible since $q > d \geqslant 2$.
	The estimate then follows from combining the ellipticity of $A$, integration by parts, and the H\"{o}lder inequality (twice) since, using that $w = 0$ on $\partial B(x_0,\,R)$,
	\begin{align*}
		\theta \norm{\nabla w}{L^2 (B(x_0,\,R) )}^2
		&\leqslant \int_{B(x_0,\,R)} A \nabla w \cdot \nabla w
		= - \int_{B(x_0,\,R)} F \cdot\nabla w
	\\
		&\leqslant \norm{F}{L^q (B(x_0,\,R) )} { \lvert B(x_0,\,R) \rvert}^{d/\tilde{q}} \norm{\nabla w}{L^2 (B(x_0,\,R) )}.
	\end{align*}
	First dividing both sides of this inequality by $\norm{\nabla w}{L^2}^2$ and then squaring yields
	$
		\norm{\nabla w}{L^2 (B(x_0,\,R)}^2 \leqslant C(d,\,q,\,\theta) \norm{F}{L^q (B(x_0,\,R)}^2 R^{2d / \tilde{q}}.
	$
	To conclude we simply note that, since $q>d$, $\alpha \vcentcolon= 1 - \frac{d}{q}$ belongs to $(0,\,1)$ such that $2d / \tilde{q} = d-2+2\alpha$ as desired.
\end{proof}

As we approach the close of the portion of this section devoted to results pertaining to the de Giorgi theory we record the following technical lemma without proof.

\begin{lemma}
\label{lemma:technical_lemma_inhom_DG}
	Let $\phi : (0,\,\infty) \to (0,\,\infty)$ be a non-decreasing map.
	Suppose that there are $0 < \beta < \gamma$ and $A,\, B,\, R_0 > 0$ such that
	$
		\phi(\rho) \leqslant A { \left( \frac{\rho}{r} \right)}^\gamma \phi (r) + B r^\beta
	$
	for every $0 < \rho \leqslant r \leqslant R_0$.
	Then, for every $\delta \in (\beta,\,\gamma)$ there exists $C = C(A,\,\beta,\,\gamma,\,\delta) > 0$ such that
	$
		\phi(\rho) \leqslant C \left[ { \left( \frac{\rho}{r} \right) }^\delta \phi(r) + B \rho^\beta \right]
	$
	for every $0 < \rho \leqslant r \leqslant R_0$.
\end{lemma}
\begin{proof}
	See Lemma 3.4 of \cite{han_lin}.
\end{proof}

The last result in this section relevant to the de Giorgi theory is the following, which is a standard elliptic estimate.

\begin{lemma}[Standard $\mathring{H}^1$ estimate for uniformly elliptic operators]
\label{lemma:std_H_1_est_unif_ell_op}
	Let $u$ be an $\mathring{H}^1$-weak solution of
	$
		-\nabla\cdot (A(x)\nabla u) = \nabla\cdot F 
	$
	in $\mathbb{T}^d$
	for some $A\in L^\infty$ which is \hyperref[def:unif_ell]{uniformly elliptic} with constants $\theta,\,\Lambda$ in $\mathbb{T}^d$ and some $F \in L^2 (\mathbb{T}^d)$.
	The following estimate holds:
	$
		\norm{\nabla u}{L^2}  \leqslant \frac{2}{\theta} \norm{F}{L^2}
	$.
\end{lemma}
\begin{proof}
	This follows from a standard energy estimate, using $u$ as a test function:
	\begin{align*}
		\theta \int_{\mathbb{T}^d} {\lvert \nabla u \rvert}^2
		\leqslant \int_{\mathbb{T}^d}  A(x)\nabla u \cdot \nabla u
		= \int_{\mathbb{T}^d} F (\nabla\cdot u)
		\leqslant \frac{2}{\theta} \norm{F}{L^2}^2 + \frac{\theta}{2} \norm{\nabla u}{L^2}^2.
	\end{align*}
	An absorbing $\varepsilon$-Cauchy inequality concludes the proof.
\end{proof}

Finally we conclude this section with a convergence result for approximations using mollifiers which is used in \fref{Section}{sec:smoothing}.

\begin{lemma}[Uniform convergence of mollification of Lipschitz functions]
\label{lemma:unif_conv_moll_Lip_func}
	Let $\varphi$ be a \hyperref[def:mollifier]{standard mollifier} and let $f:\mathbb{R}\to\mathbb{R}$ be Lipschitz.
	For
	$
		C_\varphi \vcentcolon= \int \lvert y \rvert \varphi (y) dy < \infty,
	$
	we have that
	$
		\lvert \left( f - f \ast \varphi_\varepsilon  \right) (x) \rvert \leqslant C_\varphi \text{Lip} (f) \varepsilon
	$
	for any $x\in\mathbb{R}$ and all $\varepsilon > 0$.
\end{lemma}
\begin{proof}
	Note that
	$
		f(x) - (f\ast \varphi_\varepsilon )(x)
		= \int \left[ f(x) - f(x-y) \right] \varphi_\varepsilon (y) dy
	$ since $\varphi$ is \hyperref[def:mollifier]{normalised}.
	The claim then follows from the non-negativity of $\varphi$ and the change of variable $z = y/\varepsilon$.
\end{proof}


\section{Rosetta stone}
\label{sec:rosetta}

The purpose of this section is to provide a dictionary allowing us to translate between the notation of \cite{PQG2017} which introduces the precipitating quasi-geostrophic system
and the notation of this paper.
In particular we will discuss how considering the simple version of $\PV$-and-$M$ inversion recorded in \eqref{eq:PV_M} may be done without loss of generality, qualitatively speaking.

For the purpose of this section, and of this section only, we introduce the constants $f$, $B_\theta$, $B_q$, $C_\theta$, and $C_q > 0$.
We warn the reader that they should not get too attached to these constants: they only appear in this section and are set to unity (and thus disappear) everywhere else in the paper.
They appear here for one reason: justifying that the version of $\PV$-and-$M$ inversion recorded in \eqref{eq:PV_M} is not \emph{overly} simplified.

\fref{Table}{tab:rosetta_notation} provides a dictionary between the constants introduced above and those used in \cite{PQG2017},
as well as between the notation used here and the notation used in that paper.
\begin{table}
\centering
\begin{tabular}{cc}
	\begin{tabular}{cc}
		2017 paper			& Our notation			\\\hline
		$\hat{z}$, $ \partial_z$	& $e_3$, $\partial_3 $		\\
		$\nabla_h$			& $\nabla_h$			\\
		$(u,\,v,\,0)$			& $u$				\\
		$\zeta$				& $ {(\nabla\times u)}_3$	\\
		$\phi$				& $p$				\\
		$\psi$				& $p/f$				\\
		$\theta_e$			& $\theta_e$			\\
		$q_t$				& $q$				\\
		$PV_e$			& $\PV$				\\
		$M$				& $M$				\\
	\end{tabular}
	&
	\begin{tabular}{cccc}
		2017 paper						& Our notation	\\\hline
		$f$							& $f$		\\
		$\frac{g}{\theta_0}$					& $B_{\theta_e}$	\\
		$\frac{gL_v}{\theta_0 c_p}$				& $B_q$		\\
		$\frac{d\tilde{\theta}_e}{dz}$				& $C_{\theta_e}$	\\
		$-\frac{d\tilde{q}_t}{dz} = - \frac{d\tilde{q}_v}{dz}$	& $C_q$		\\
		(as $\tilde{q}_t = \tilde{q}_v$)					\\
	\end{tabular}
\end{tabular}
\caption{
	Each table compares the notation used in \cite{PQG2017} (in its left column) and the notation used in this paper (in its right column).
	Note that the constants appearing in the right column of the table on the right-hand side are only used in \fref{Appendix}{sec:rosetta} and are all set to unity elsewhere.
}
\label{tab:rosetta_notation}
\end{table}
With this dictionary in hand we can compare the forms taken by various expressions and identities, depending on the notation used.
This is recorded in \fref{Table}{tab:rosetta_identities}, which contains expressions for geostrophic and hydrostatic balance, as well as for the potential vorticity, the moist variable $M$,
and the saturated and unsaturated buoyancy frequencies (introduced in \cite{PQG2017}).

\begin{table}
\centering
\begin{tabular}{ccc}
	2017 paper												& Our notation					& Constants set to unity\\\hline
	$ f \hat{z} \times (u,\,v,\,0) = - \nabla_h \phi$							& $u = \frac{1}{f} \nabla_h^\perp h$	& $u = \nabla_h^\perp p$\\
	$ f \frac{\partial\psi}{\partial z} = \frac{g}{\theta_0} \theta_e - \frac{gL_v}{\theta_0 c_p} q_t H_u$	& $\partial_3 p = B_{\theta_e} \theta_e - B_q q H_u$	& $\partial_3 p = \theta_e - qH_u$\\
	$ PV_e = \zeta + \frac{f}{d\tilde{\theta}_e / dz} \frac{\partial\theta_e}{\partial z}$			& $\PV = {(\nabla\times u)}_3 + \frac{f}{C_{\theta_e}} \partial_3 \theta_e$	
																		& $\PV = {(\nabla\times u)}_3 + \partial_3 \theta_e$\\
	$ M = q_t - \frac{d \tilde{q}_t / dz}{d \tilde{\theta}_e / dz} \theta_e	$				& $M = q + \frac{C_q}{C_{\theta_e}} \theta_e$		& $M = \theta_e + q$\\
	$ N_u^2 = \frac{g}{\theta_0} \frac{d\tilde{\theta}_e}{dz} - \frac{gL_v}{\theta_0 c_p} \frac{d \tilde{q}_t}{dz}$
														& $N_u^2 = B_{\theta_e} C_{\theta_e} + B_q C_q$		& $N_u^2 = 2$\\
	$ N_s^2 = \frac{g}{\theta_0} \frac{d\tilde{\theta}_e}{dz}$						& $N_s^2 = B_{\theta_e} C_{\theta_e}$			& $N_s^2 = 1$\\
\end{tabular}
\caption{
	Comparisons of various identities and expressions using the notation of \cite{PQG2017} (on the left),
	the notation of \fref{Appendix}{sec:rosetta} (in the middle), and the notation of the everywhere else in this paper (on the right).
	From top to bottom: geostrophic balance, hydrostatic balance, (equivalent) potential vorticity, moist variable $M$, the unsaturated buoyancy frequency, and the saturated buoyancy frequency.
}
\label{tab:rosetta_identities}
\end{table}

With all of these notational equivalences in hand, we may formulate $\PV$-and-$M$ inversion in six different ways: two for each of the three sets of notation;
one with the Heavisides $H_u$ and $H_s$ and one invoking the function $\min_0 = \min ( \,\cdot\, ,\, 0)$.
From top to bottom we use the notation of \cite{PQG2017}, the notation of this section, and the notation used everywhere else in this paper
(which is the same as the notation used in this section, albeit with the constants $f$, $B_{\theta_e}$, $B_q$, $C_{\theta_e}$, and $C_q$ set to unity):
\begin{align}
	PV_e
	&= \nabla^2_h \psi + f^2 \frac{\partial}{\partial z} \left[
		\frac{1}{N_u^2} \left( \frac{\partial\psi}{\partial z} + \frac{g L_v}{\theta_0 f c_p} M  \right) H_u
		+ \frac{1}{N_s^2} \left( \frac{\partial\psi}{\partial z} \right) H_s 
	\right]
\nonumber\\
	&= \nabla^2_h \psi + f^2 \frac{\partial}{\partial z} \left[
		\frac{1}{N_s^2} \frac{\partial\psi}{\partial z} + \left( \frac{1}{N_s^2} - \frac{1}{N_u^2} \right) {\min}_0\, \left(
			\frac{-g}{f\theta_0} \cdot \frac{d\tilde{\theta}_e / dz}{d\tilde{q}_t / dz} M - \partial_3 p
		\right)
	\right]
\nonumber\\
	\PV
	&= \frac{1}{f} \Delta_h p + f \left[
		\frac{1}{N_u^2} \left( \partial_3 p + B_q M \right) H_u + \frac{1}{N_s^2} \left( \partial_3 p \right) H_s
	\right]
\label{eq:PV_M_min0_constants}\\
	&= \frac{1}{f} \Delta_h p + f \partial_3 \left[
		\frac{1}{N_s^2} \partial_3 p + \left( \frac{1}{N_s^2} - \frac{1}{N_u^2} \right) {\min}_0\, \left( \frac{B_{\theta_e} C_{\theta_e}}{C_q} M - \partial_3 p \right)
	\right]
\nonumber\\
	\PV
	&= \Delta_h p + \partial_3 \left[ \frac{1}{2} (\partial_3 p + M) H_u + (\partial_3 p) H_s \right]
\nonumber\\
	&= \Delta_h p + \partial_3 \left[ \partial_3 p + \frac{1}{2} {\min}_0\, \left( M - \partial_3 p \right) \right]
	= \Delta p + \frac{1}{2} \partial_3 {\min}_0\, \left( M - \partial_3 p \right)
\nonumber
\end{align}

In particular we note that the equations above follow from the dictionary recorded in \fref{Table}{tab:rosetta_identities}.
The only intermediate computation required is recorded below, without proof.
\begin{lemma}
\label{lemma:corresp}
	Suppose that, for some strictly positive constants $B_{\theta_e}$, $B_q$, $C_{\theta_e}$, and $C_q$ the identities
	$
		\partial_3 p = B_{\theta_e} \theta_e - B_q qH_u
	$ and $
		M = q + \frac{C_q}{C_{\theta_e}} \theta_e
	$
	hold.
	Then we have the identity
	$
		q = C_q \left(
			\frac{1}{N_u^2} H_u + \frac{1}{N_s^2} H_s
		\right) \left(
			\frac{B_{\theta_e} C_{\theta_e}}{C_q} M - \partial_3 p
		\right)
	$
	as well as the identities given by
	\begin{align*}
		\frac{\theta_e}{C_{\theta_e}}
		&= \frac{1}{N_u^2} \left( \partial_3 p + B_q M \right) H_u + \frac{1}{N_s^2} \left( \partial_3 p \right) H_s
	\\
		&= \frac{1}{N_s^2} \partial_3 p + \left( \frac{1}{N_s^2} - \frac{1}{N_u^2} \right) {\min}_0 \left( \frac{B_{\theta_e} C_{\theta_e}}{C_q} M - \partial_3 p \right)
	\end{align*}
	where $N_u^s$ and $N_s^2$ are defined in \fref{Table}{tab:rosetta_identities} and where, as usual, $\min_0 (x) \vcentcolon= \min(x,\,0)$.
\end{lemma}

In particular, since elsewhere in this paper the constants $f$, $B_{\theta_e}$, $B_q$, $C_{\theta_e}$, and $C_q$ are set to unity, we obtain the following expression for the potential temperature $\theta_e$ and the water content $q$.
This expression for $q$ is essential in taking $\PV$-and-$M$ inversion from its original form \eqref{eq:PV_M} to its reformulation \eqref{eq:PV_M_min0} which is used throughout this paper.

\begin{cor}
\label{cor:q}
	Suppose that $\partial_3 p = \theta_e - q H_u$ and $M = \theta_e + q$. Then we have the identities $\theta_e = \partial_3 p + \frac{1}{2} \min \left( M - \partial_3 p,\, 0 \right)$
	and  $q = \left( \frac{1}{2} H_u + H_S \right) \left( M - \partial_3 p \right)$.
	In particular the sign of $q$ agrees with the sign of $M-\partial_3 p$.
\end{cor}
\begin{proof}
	This follows from setting $B_{\theta_e} = B_q = C_{\theta_e} = C_q = 1$ in \fref{Lemma}{lemma:corresp}.
\end{proof}

We conclude this section by noting that the form of $\PV$-and-$M$ inversion recorded in \eqref{eq:PV_M} and studied throughout this paper
has nothing to envy to the more general version \eqref{eq:PV_M_min0_constants} recorded above.
The latter is \emph{quantitatively} more general, but the only \emph{qualitative} relation that matters between the various constants at play in \eqref{eq:PV_M_min0_constants}
is that $N_u^2 > N_s^2$ (which is always the case since $N_u^2 - N_s^2 = B_q C_q > 0$).
By setting all of the constants $f$, $B_{\theta_e}$, $B_q$, $C_{\theta_e}$, and $C_q$ to unity (as is done elsewhere in this paper) we are therefore studying the case where $N_u^2 = 2$ and = $N_s^2 = 1$.
As far as the well-posedness theory and the iterative methods studied in this paper are concerned,
this is computationally simpler while describing the same qualitative behaviour as any other choice of values for the constants
$f$, $B_{\theta_e}$, $B_q$, $C_{\theta_e}$, and $C_q$.


\section{Comparison of the variational energy and the conserved energy}
\label{sec:comparison}

In this section we discuss how the variational energy introduced in this paper in order to characterise $\PV$-and-$M$ inversion (c.f. \fref{Definition}{def:energy})
differs from the energy that is conserved by the dynamics of the precipitating quasi-geostrophic equations.

Indeed, while this conserved energy may at first glance seem like a viable candidate for the variational formulation of $\PV$-and-$M$ inversion, we will see that it is not the case.
We will also see that the energy we introduce in \fref{Definition}{def:energy} has better regularity properties than the conserved energy.
First we write the conserved energy in terms of solely $p$ and $M$.

\begin{lemma}[Rewriting the conserved energy]
\label{lemma:rewriting_conserved_energy}
	The conserved energy $E_{cons}$ introduced in \cite{marsico} as $E_1$ may, in the notation of \fref{Appendix}{sec:rosetta}, be written as follows:
	\begin{equation}
	\label{eq:cons_en_u_theta_q}
		E_{cons} = \int_{\mathbb{T}^3} \frac{1}{2} {\lvert u \rvert}^2 + \left(
				\frac{\widetilde{B}_{\theta_e}}{2 N_u^2} H_u + \frac{B_{\theta_e}}{2 N_s^2} H_s
			\right) \theta_e^2
			+ \frac{\widetilde{B}_q}{2 N_s^2} q^2 H_u
	\end{equation}
	for $\widetilde{B}_{\theta_e} = B_{\theta_e} \left( B_{\theta_e} + B_q C_q / C_{\theta_e} \right)$ and $\widetilde{B}_q = B_q \left( B_q + B_{\theta_e} C_{\theta_e} / C_q \right)$.
	In particular, if we consider $B_{\theta_e} = B_q = C_{\theta_e} = C_q = 1$ then
	$
		E_{cons} = \int_{\mathbb{T}^3} \frac{1}{2} {\lvert u \rvert}^2 + \frac{1}{2} \theta_e^2 + \frac{1}{2} q^2 H_u
	$
	such that, if geostrophic balance $u = \nabla_h^\perp p$ and hydrostatic balance $\partial_3 p = {\theta_e} - qH_u$ hold,
	\begin{equation}
	\label{eq:cons_en_p_M}
		E_{cons} = \int_{\mathbb{T}^3} \frac{1}{2} {\lvert \nabla p \rvert}^2 + \frac{1}{4} (M + \partial_3 p) {\min}_0\, (M - \partial_3 p)
	\end{equation}
	where as usual, $\min_0 (x) \vcentcolon= \min(x,\,0)$.
\end{lemma}
\begin{proof}
	We recall from \cite{marsico} that the conserved energy is defined to be
	\begin{equation*}
		E_{cons}
		\vcentcolon= \int_{\mathbb{T}^3} \frac{1}{2} {\lvert u \rvert}^2 + \frac{b_u^2}{2 N_u^2} H_u + \frac{b_s^2}{2 N_s^2} H_s
			+ \frac{N_u^2 N_s^2}{2 (N_u^2 - N_s^2)} \widetilde{M}^2 H_u
	\end{equation*}
	where $b_u = \theta_e - q$ and $b_s = \theta_e$ such that the buoyancy may be decomposed as $\theta_e - qH_u = b_u H_u + b_s H_s$
	and where $\widetilde{M}$ is a constant multiple of the moist variable $M$ used herein. Indeed we have that
	$
		\widetilde{M}
		\vcentcolon= b_u / N_u^2 - b_s / N_s^2
		= -(B_q / N_u^2 ) M.
	$
	A straightforward but tedious computation then produces \eqref{eq:cons_en_u_theta_q}.
	To obtain \eqref{eq:cons_en_p_M} we use \fref{Lemma}{lemma:corresp} to write $\theta_e$ and $q$ in terms of $M$ and $\partial_3 p$ and plug into the expression for the conserved energy.
	A short computation then verifies that
	\begin{align*}
		{\lvert u \rvert}^2 + \theta_e^2 + q^2 H_u
		= {\lvert \nabla_h p \rvert}^2 + { \left( \partial_3 p + \frac{1}{2} ( M - \partial_3 p)  H_u \right)}^2 + \frac{1}{4} {(M - \partial_3 p)}^2 H_u
	\\
		= {\lvert \nabla p \rvert}^2 + \frac{1}{2} (M + \partial_3 p) {\min}_0\, (M - \partial_3 p),
	\end{align*}
	as desired.
\end{proof}

We may then use \eqref{eq:cons_en_p_M} to compute the Euler-Lagrange equation associated with the conserved energy.

\begin{cor}[Euler-Lagrange equation associated with the conserved energy]
\label{cor:Euler_Lagrange_cons_en}
	Let $E_{cons}$ be as defined as in \fref{Lemma}{lemma:rewriting_conserved_energy}, with $M$ being fixed and viewing $E_{cons}$ as a function of the pressure $p$ only.
	The Euler-Lagrange equation associated with $E_{cons}$ is
	$
		-\Delta p + \frac{1}{2} \partial_3 ( (\partial_3 p) H_u ) = 0.
	$
\end{cor}
\begin{proof}
	We simply take a derivative of $E_{cons} (p)$ in the direction of $\phi$, obtaining
	$
		DE_{cons} (p) \phi = \int_{\mathbb{T}^3} \nabla p \cdot \nabla\phi - \frac{1}{2} (\partial_3 p) H_u ( \partial_3 \phi),
	$
	from which the claim follows.
\end{proof}

\begin{remark}[Comparison of the variational energy and the conserved energy]
\label{rmk:comparison_en}
	Combining \fref{Lemma}{lemma:Gateaux_diff_energy} and \fref{Corollary}{cor:Euler_Lagrange_cons_en} tells us that
	\begin{equation}
	\label{eq:compare_en}
		DE(p) \phi = \underbrace{DE_{cons} (p) \phi}_{\sim\, \text{principal part}} + \underbrace{\langle \PV - \frac{1}{2} \partial_3 (MH_u) ,\, p \rangle}_{\text{nonlinear forcing}}.
	\end{equation}
	What is very interesting about this identity is that, in some sense, the conserved energy extracts the principal part of the operator characterising $\PV$-and-$M$ inversion.
	Indeed, we can essentially view $\PV - \frac{1}{2} \partial_3 (MH_u)$ as a term forcing the Euler-Lagrange operator associated with the conserved energy.
	However, and this is crucial, this forcing term is \emph{nonlinear} in $p$ due to the presence of $H_u$.
	The derivative of $ \frac{1}{2} \langle \partial_3 (MH_u),\, p \rangle$ with respect to $p$ is \emph{not} equal to $ \frac{1}{2} \partial_3 ( MH_u)$ itself
	which means
	that we cannot simply claim that $\PV$-and-$M$ inversion is the Euler-Lagrange equation associated with $E_{cons} + \PV - \frac{1}{2} \partial_3 (MH_u)$.

	This means that, in order to identify a variational formulation of $\PV$-and-$M$ inversion it is not sufficient to take inspiration from the conserved energy.
	We must instead analyse the PDE carefully, which leads to the definition of a \emph{different} energy.
	The identity \eqref{eq:compare_en} also highlights the essential trick underlying the definition of the variational energy $E$:
	the forcing term $ \frac{1}{2} \partial_3 (M H_u)$ is carefully folded into the energy, thus obtaining a valid variational formulation of $\PV$-and-$M$ inversion.

	In particular this leads to a more regular energy.
	Indeed a comparison of \eqref{eq:def_en} and \eqref{eq:cons_en_p_M} shows that the energy density appearing in the variational energy is $C^{1,\,1}$, as a function of $\partial_3 p$,
	whereas the energy density appearing in the conserved energy is only $C^{0,\,1}$.
\end{remark}


\section{Precipitating quasi-gesotrophic equations}
\label{sec:pqg}

Since this paper introduces the precipitating quasi-gesotrophic equations to an audience that may not be familiar with that system,
we record it here for the reader's convenience.
The dynamic equations are, in their simplest form

\begin{equation*}
	\left\{
	\begin{aligned}
		& ( \partial_t + u\cdot \nabla ) \PV = - \partial_3 u \cdot \nabla \theta_e \text{ and } \\
		& ( \partial_t + u\cdot \nabla ) M = \partial_3 \left( q^+ \right),
	\end{aligned}
	\right.
\end{equation*}
where $q^+ = \max (q,\,0)$ denotes the positive part of $q$.
Here $\PV$, $M$, $\theta_e$, $q$, and $p$ are scalar fields whereas $u$ is a vector field
and we recall that the spatial domain is the three-dimensional torus.
The unknowns $u$, $\theta_e$ and $q$ are determined by $p$ and $M$: $u$ is determined by the geostrophic balance $u_h = \nabla_h^\perp p$,
while $\theta_e$ and $q$ can be recovered from inverting the hydrostatic balance $\theta_e - \min(0,\,q) = \partial_3 p$ and the definition of the moist variable $M \vcentcolon= \theta_e + q$.
Note that $u$ is purely horizontal since its vertical component vanishes, but still depends on all three spatial coordinates.
The pressure $p$ is in turn determined by the potential vorticity $\PV$ and $M$ as the unique solution of the $\PV$-and-$M$ inversion
\begin{equation*}
	\Delta p + \frac{1}{2} \partial_3 \left( \min \left( M - \partial_3 p,\, 0 \right) \right) = \PV.
\end{equation*}
In particular, by the solvability of $\PV$-and-$M$ inversion established in this paper we see that $u$, $\theta_e$, and $q$ are entirely determined by $\PV$ and $M$.
We can therefore view the full system as a nonlinear and nonlocal transport equation for $\PV$ and $M$.
Alternatively, akin to what is done for dry QG, we can view the entire system as dependent solely on $p$ and $M$.

For the sake of comparison, we note that in the dry case the geostrophic and hydrostatic balances tell us that $u = \nabla_h^\perp p$ and $\theta_e = \theta = \partial_3 p$
and hence the forcing term $-\partial_3 u \cdot \nabla\theta_e$ in the dynamic equation for the potential vorticity vanishes.
This is as expected since the potential vorticity is purely advected in the dry case.

The forcing term $\partial_3 \left(q^+\right)$ in the dynamic equation for the moist variable $M$ is the eponymous precipitation term.
Indeed, when saturation is taken to occur at $q=0$ as is done here, $q^+$ describes precisely \emph{liquid rainwater}, which is assumed in this model to fall directly downward at a constant speed.
For simplicity all physical constants have been set to unity -- the appropriate constants can be found in \cite{PQG2017}.


\section*{Acknowledgments}
This work was partially supported by the National Science Foundation under grant NSF-DMS-1907667.
As recipient of an INI-Simons Postdoctoral Research Fellowship, A.R.-T. gratefully acknowledges support from
the Simons Foundation, the Isaac Newton Institute for Mathematical Sciences, and the Department of Applied Mathematics and Theoretical Physics of the University of Cambridge.

The authors thank the anonymous reviewers for helpful comments and thank Peter Constantin for helpful discussions in the early stages of this work.

On behalf of all authors, the corresponding author states that there is no conflict of interest.


\bibliographystyle{alpha-bis}
\bibliography{references}

\end{document}